\documentclass[reqno]{amsart}
\usepackage{amsmath}
\usepackage{amsfonts}
\usepackage{amssymb}
\usepackage{enumerate}
\usepackage{amstext}
\usepackage{amsbsy}
\usepackage{amsopn}
\usepackage{bbm,amsthm}
\usepackage{amscd}
\usepackage[pdftex]{color}
\usepackage{amsxtra}
\usepackage{upref}
\usepackage{graphicx}
\usepackage{epstopdf}
\DeclareFontFamily{OML}{rsfs}{\skewchar\font'177}
\DeclareFontShape{OML}{rsfs}{m}{n}{ <5> <6> rsfs5 <7> <8> <9> rsfs7
  <10> <10.95> <12> <14.4> <17.28> <20.74> <24.88> rsfs10 }{}
\DeclareMathAlphabet{\mathfs}{OML}{rsfs}{m}{n}

\newcommand{\x}{\ensuremath{\underline{x}}}
\newtheorem{thm}{Theorem}[section]
\newtheorem{lem}[thm]{Lemma}

\newtheorem{cor}[thm]{Corollary}

\newtheorem*{theorem*}{Theorem}

\newtheorem*{example*}{Example}

\numberwithin{equation}{section}

\def\var{\textrm{var}}

\renewcommand{\epsilon}{\varepsilon}
\def\wt{\widetilde}

\def\text#1{\textrm{#1}}

\def\emptyset{\varnothing}

\def\e{\epsilon}

\def\vf{\varphi}

\def\l{\lambda}

\def\s{\sigma}

\def\x{\times}
\def \R{\mathbb R}
\def \N{{\mathbb N}}
\def\E{\mathbb E}
\def \Z{\mathbb Z}

\def\ov{\overline}
\def\un{\underline}

\def\Q{\mathbb Q}
\def\wh{\widehat}
\def\({\biggl(}
\def\){\biggr)}

\def\<{\mathbf\langle}
\def\>{\mathbf\rangle}

\DeclareMathOperator\const{const}

\def\loc{\mathrm{loc}}
\title[Ergodic properties of equilibrium measures for 3-dim flows]{Ergodic properties of equilibrium measures for smooth three dimensional flows}
\author{Fran\c{c}ois Ledrappier, Yuri Lima, and Omri Sarig }
\date{October 26, 2015}
\address{F. Ledrappier\\ Sorbonne Universit\'es\\ UPMC Univ. Paris 06\\UMR 7599\\ LPMA\\ F-75005\\ Paris\\ France
and CNRS\\ UMR 7599\\ LPMA\\ F-75005\\ Paris\\ France}
\email{fledrapp@nd.edu}
\address{Y. Lima\\ Laboratoire de Math\'ematiques d'Orsay\\ Univ. Paris-Sud\\ CNRS\\ Universit\'e
Paris-Saclay\\ 91405 Orsay\\ France}
\email{yurilima@gmail.com}
\address{O. Sarig\\ Faculty of Mathematics and Computer Science\\ The Weizmann Institute of Science\\ POB 26, Rehovot, Israel}
\email{omsarig@gmail.com}

\begin{document}
\maketitle
\begin{abstract}
Let $\{T^t\}$ be a smooth flow with positive speed and positive topological entropy on a compact smooth three dimensional manifold, and
let $\mu$ be an ergodic measure of maximal entropy. We show that either  $\{T^t\}$ is Bernoulli, or $\{T^t\}$ is isomorphic to the product of a Bernoulli flow and a rotational flow. Applications are given to Reeb flows.
\end{abstract}

\section{Introduction and statement of main results}
\subsection*{Introduction} In 1973, Ornstein and Weiss proved  that the geodesic flow of a compact smooth surface with constant negative curvature is Bernoulli with respect to the Liouville measure \cite{Ornstein-Weiss-Geodesic-Flows}.
Ratner extended this to variable negative curvature \cite{Ratner-Flows-Bernoulli}. In the case of  non-positive and non identically zero curvature, Pesin showed that all ergodic components of the Liouville measure are Bernoulli
\cite{Pesin-Characteristic-1977}, \cite[Thm 12.2.13]{Barreira-Pesin-Non-Uniform-Hyperbolicity-Book}.
It follows from his work that all other ergodic components (if they exist) have zero entropy. Katok and Burns extended Pesin's work to Reeb flows \cite{Katok-Burns-Infinitesimal}. Burns and Gerber proved that geodesic flows on certain surfaces with some positive curvature (``Donnay's examples") are Bernoulli \cite{Burns-Gerber}.
Hu, Pesin and Talitskaya constructed smooth volume-preserving Bernoulli flows
on every compact manifold of dimension at least three \cite{Hu-Pesin-Talitskaya-2004}.

Ratner's work  extends to general Anosov flows equipped with ergodic equilibrium measures of H\"older continuous potentials  \cite{Ratner-Flows-Bernoulli}. In this case the flow is either Bernoulli, or isomorphic to a Bernoulli flow times a rotational flow (this happens in the non-mixing case).
Pesin's work  extends to all  $C^{1+\epsilon}$ flows preserving an ergodic hyperbolic measure whose conditional measures on the unstable manifolds are absolutely continuous with respect to the induced Riemannian measure \cite{Pesin-Izvestia-1976}, \cite{Ornstein-Weiss-General-Bernoulli}, \cite{Katok-Strelcyn}, \cite{Ledrappier-Mesures-de-Sinai}, with the same modification in the non-mixing case.

The measure of maximal entropy does not have absolutely continuous conditional measures, except in special cases \cite{Katok-Closed-Geodesics}.
The purpose of this paper is to determine the ergodic theoretic structure of this measure in the context of general smooth three dimensional flows with positive topological entropy.
Our methods also apply to ergodic equilibrium measures of H\"older potentials with positive entropy.

\subsection*{Basic definitions} Let $(X,\mathfs B,\mu)$ be a Lebesgue probability space.

\medskip
\noindent
{\sc Measurable flow:} A quadruple $\textsf{T}=(X,\mathfs B,\mu,\{T^t\})$ such that $(t,x)\mapsto T^t(x)$ is measurable, and the {\em time--$t$ map}
$(X,\mathfs B,\mu,T^t)$ is probability preserving, $\forall t\in\R$.

\medskip
\noindent
{\sc Eigenfunction:} A non-constant measurable function $f$ is an {\em eigenfunction} of $\textsf{T}$ (with eigenvalue $e^{i\alpha}$)
if for a.e. $x\in X$, $f(T^t x)=e^{i\alpha t}f(x)$ for all $t\in\R$.
$\textsf{T}$ is called {\em ergodic} if $1$ is not an eigenvalue, and {\em weak-mixing} if it has no eigenvalues at all.

\medskip
\noindent
{\sc Entropy:} The {\em entropy} of $\textsf{T}$ is the entropy of the time--1 map $T^1$.

\medskip
\noindent
{\sc Rotational flow:} \label{Page-rotational-flow}
Given $c>0$, the {\em rotational flow} is $T^t(x):=x+t/c \text{ (mod 1)}$ on $\R/\Z$
equipped with the Haar measure. $c$ is called the {\em period}, and it is an invariant of the flow since $c=\min\{t>0: T^t={\rm Id}\}$.

\medskip
\noindent
{\sc Bernoulli flow:} $\textsf{T}$ is called {\em Bernoulli} if $T^1$ is a Bernoulli automorphism.
$\textsf{T}$ is called {\em Bernoulli up to a period} if $\textsf{T}$ is Bernoulli, or if $\textsf{T}$
is isomorphic to the product of a Bernoulli flow and a rotational flow.

\medskip
If $\textsf{T}$ is a Bernoulli flow then $T^t$ is a Bernoulli automorphism, $\forall t\neq 0$ \cite{Ornstein-Imbedding-in-Flows}.
Entropy is a complete set of invariants for Bernoulli flows \cite{Ornstein-B-Flows}, and entropy and period
(if it exists) are a complete set of invariants for Bernoulli up to a period flows since the Bernoulli term is determined
by the entropy and the rotational term is the Pinsker factor, see \cite[Prop. 4.4]{Thouvenot-Handbook}.

\subsection*{Main results}
Let $M$ be a three dimensional compact $C^\infty$ Riemannian manifold without boundary, let $\mathfs B$
be its Borel $\sigma$--algebra,
let $X: M\to TM$ be a $C^{1+\epsilon}$ vector field on $M$ s.t. $X_p\neq 0$, $\forall p\in M$, let $\textsf{T}$ be the flow on
$M$ generated by $X$, and let $\mu$ be a $\textsf{T}$--invariant probability measure.

\medskip
\noindent
{\sc Equilibrium measure:} $\mu$ is an {\em equilibrium measure} of a
{\em potential} $F:M\to\R$ if $h_\mu(T^1)+\int_M F d\mu=\sup\{h_\nu(T^1)+\int_M F d\nu\}$, where sup ranges over
all $\textsf T$--invariant probability measures $\nu$.
If $F= 0$, then $\mu$ is called a {\em measure of maximal entropy}.

\medskip
Equilibrium measures always exist if $X$ is $C^\infty$ and $F$ is continuous \cite{Newhouse-Entropy}.

\begin{thm}\label{Thm-Main}
Under the above assumptions on $M, X,\textsf{T}$,
every equilibrium measure of a H\"older continuous potential has at most countably many ergodic components with positive entropy.
Each of them is Bernoulli up to a period.
\end{thm}

Periods can exist (e.g. for the constant suspension of an Anosov diffeomorphism), but
sometimes they can be discounted.
Let $\{T^t\}$ be a Reeb flow on a compact smooth three dimensional contact manifold $M$ (see \S\ref{Section-Reeb} for definitions).
For example, $\{T^t\}$ could be the geodesic flow of a surface, or the Hamiltonian flow of a system with two degrees of freedom on a regular energy surface \cite{Abraham-Marsden-Mechanics}. Katok and Burns showed that every ergodic absolutely continuous invariant measure with positive entropy is Bernoulli \cite{Katok-Burns-Infinitesimal}. The following result covers other  measures of interest,  such as the measures of maximal entropy.

\begin{thm}\label{Thm-Reeb}
If $\textsf{T}$ is a three dimensional Reeb flow, then every equilibrium measure of a H\"older continuous potential has at most countably
many ergodic components with positive entropy. Each of them is Bernoulli.
\end{thm}

%
%


\begin{cor}\label{Corollary-Surfaces}
Let $S$ be a compact smooth orientable surface without boundary, with nonpositive and non-identically zero curvature.
Then the geodesic flow of $S$ is Bernoulli with respect to its (unique) measure of maximal entropy.
\end{cor}

\begin{proof}
Let $m$ be the invariant Liouville measure. By the curvature assumptions, $m$ has positive metric entropy, see
for example \cite[Corollary 3]{Pesin-Entropy-Formulas}. Hence the geodesic flow has positive topological entropy.
Also by the curvature assumptions, $S$ is a rank one manifold \cite{Ballmann-Brin-Eberlein}, therefore
there is a unique measure of maximal entropy \cite{Knieper-Rank-One-Entropy}. By uniqueness, it is ergodic.
By Theorem \ref{Thm-Reeb}, it is Bernoulli.
\end{proof}


The ``geometric potential" $J(x):=-\frac{d}{ds}|_{s=0}\log\|dT^s|_{E^u(x)}\|$  and its scalar multiples (see \cite{Bowen-Ruelle-SRB} and \S\ref{Section-Geometric-Potential}) are not directly covered by Theorems \ref{Thm-Main} and \ref{Thm-Reeb}, because they are  not necessarily H\"older continuous or even globally defined on $M$. But our methods do apply to them and give the following:

\begin{thm}\label{Thm-Geometric-Potential}
Under the assumptions of Theorem \ref{Thm-Main},
every equilibrium measure of $tJ$ $(t\in\R)$ has at most countably many
ergodic components with positive entropy. Each is Bernoulli up to a period.
If $\textsf{T}$ is  a Reeb flow, each  is Bernoulli.
\end{thm}

\begin{cor}{\cite[Thm 9.7]{Pesin-Characteristic-1977}}
Let $S$ be a compact smooth orientable surface without boundary, with nonpositive and non-identically
zero curvature. Then the geodesic flow of $S$ is Bernoulli with respect to every positive entropy ergodic
component of the invariant Liouville measure. There are at most countably many such components.
\end{cor}

\begin{proof}
The invariant Liouville measure is an equilibrium measure for the
geometric potential $J(x)$, by the Pesin Entropy Formula and the Ruelle Entropy Inequality. It has positive metric entropy,
as shown in the proof of Corollary \ref{Corollary-Surfaces}.
\end{proof}

\subsection*{Methodology}
Our approach is similar to that of \cite{Ratner-Flows-Bernoulli,Ratner-Flows-K}: First we code the flow as a
topological Markov flow (H\"older suspension of a topological Markov shift), and then we analyze equilibrium measures for the symbolic model.
The first step was done in \cite{Lima-Sarig}. The second step is the subject of the present work.

The ergodic behavior of equilibrium measures on topological Markov flows depends on the height function $r$.
If $r$ is cohomologous to a function taking values in a discrete subgroup, then one can choose a coding with constant height function,
and deduce that the flow is isomorphic to the product of a Bernoulli flow and a rotational flow. If $r$ is not cohomologous to a function
taking values in a discrete subgroup, then one can exhibit a generating sequence of very weak Bernoulli partitions as in
\cite{Ornstein-Weiss-Geodesic-Flows,Ratner-Flows-Bernoulli}, and conclude that the flow is Bernoulli. An important
step in the proof of the very weak Bernoulli property is to prove the K property.
This is done using the method of Gurevi\v{c} \cite{Gurevich-K-flows}.

In Ratner's case the flow is Anosov, and the symbolic flow is a suspension over a topological Markov shift with finite alphabet \cite{Ratner-MP-n-dimensions}.
In our case the flow is a general $C^{1+\epsilon}$ flow on a three dimensional manifold, and the topological
Markov shift has countable alphabet \cite{Lima-Sarig}. The thermodynamic formalism for countable Markov shifts \cite{Buzzi-Sarig} provides us with
the local product structure we need to implement the ideas of \cite{Gurevich-K-flows,Ornstein-Weiss-Geodesic-Flows,Ratner-Flows-Bernoulli,Ratner-Flows-K}.

The paper is divided into two parts. The first contains the analysis of topological Markov flows. The second contains the application to smooth flows, and in particular to Reeb flows and geodesic flows.

\vspace{.5cm}
\noindent
{\bf\large Part 1. Topological Markov Flows}

\section{Topological Markov Flows}

\subsection*{Topological Markov shifts (TMS)} Let $\mathfs G$ be a directed graph with countable set of vertices $V$.
We write $v\to w$ if there is an edge from $v$ to $w$. We assume throughout that  for every $v$  there
are $u,w$ s.t. $u\to v, v\to w$, and that $\mathfs G$ is not a cycle.

\medskip
\noindent
{\sc Topological Markov shift (TMS):} The {\em topological Markov shift} (TMS) associated to $\mathfs G$ is the discrete-time topological dynamical system $\sigma:\Sigma\to\Sigma$ where
$$
\Sigma=\Sigma(\mathfs G):=\{\textrm{paths on }\mathfs G\}=\{\{v_i\}_{i\in\Z}:v_i\to v_{i+1},\forall i\in\Z\},
$$
and $\sigma:\{v_i\}_{i\in\Z}\mapsto \{v_{i+1}\}_{i\in\Z}$ is the {\em left shift}.

\medskip
Points in $\Sigma$ will be denoted by $x=\{x_i\}_{i\in\Z}$.
The topology of $\Sigma$ is given by the   metric $d(x,y):=\exp[-\min\{|n|:x_n\neq y_n\}]$.
The Borel $\sigma$-algebra $\mathfs B(\Sigma)$ is  generated by the {\em cylinders}
$$
{_m[}a_0,\ldots,a_{n-1}]:=\{x\in\Sigma:x_{i+m}=a_i\text{ for all }i=0,\ldots,n-1\}.
$$
The index $m$ denotes the left-most coordinate of the constraint. If it is zero, we will simply write
$[\un{a}]:={_0[}\un{a}]$. $n$ is called the {\em length} of the cylinder, also denoted by $|\un{a}|$.
A cylinder is non-empty iff $a_0\to\cdots\to a_{n-1}$ is a path on $\mathfs G$. In this case we call the word $\un{a}$ {\em admissible}.

For $x\in\Sigma$ and $i<j$ in $\Z$, let $x_i^j:=(x_i,\ldots,x_j)$, $x_i^\infty:=(x_i,x_{i+1},\ldots)$, and $x_{-\infty}^i:=(\ldots,x_{i-1},x_i)$.

A TMS is {\em topologically transitive} iff for every $u,v\in V$ there is a finite path on $\mathfs G$ from $u$
to $v$. It is {\em topologically mixing} iff for every $u,v\in V$ there is
$N=N(u,v)$ s.t. for every $n\geq N(u,v)$ there is a path of length $n$  on $\mathfs G$ from $u$ to $v$.

Every ergodic $\sigma$--invariant probability measure on $\Sigma$ is carried by a topologically transitive
TMS inside $\Sigma$. If the measure is mixing, then the TMS is topologically mixing.

Every topologically transitive TMS has a {\em spectral decomposition}
$\Sigma=\biguplus_{i=0}^{p-1}\Sigma_i$ where each $\Sigma_i$ is the union of cylinders of length one
at the zeroth position, $\sigma^p:\Sigma_i\to\Sigma_i$ is topologically conjugate to a topologically
mixing TMS for every $i$, and $\sigma(\Sigma_i)=\Sigma_{i+1\text{(mod }p)}$  \cite{Kitchens-Book}.

\subsection*{Topological Markov flows (TMF)}
Let $r:\Sigma\to\R^+$ be H\"older continuous, bounded away from zero and infinity, and let
$\Sigma_r:=\{(x,t):x\in\Sigma,0\leq t<r(x)\}$.

\medskip
\noindent
{\sc Topological Markov flow (TMF):} The {\em topological Markov flow} (TMF) with
{\em roof function} $r$ and {\em basis} $\sigma:\Sigma\to\Sigma$  is the flow
$\{\sigma_r^\tau\}$ on $\Sigma_r$ which increases the $t$ coordinate at unit speed
subject to the identifications $(x, r(x))\sim (\sigma(x),0)$.

\medskip
Formally, $\sigma_r^\tau$ is defined as
$\sigma_r^\tau(x,t):=(\sigma^n(x),t+\tau-r_n(x))$ for the unique $n\in\Z$ s.t. $0\leq t+\tau-r_n(x)<r(\sigma^{n}(x))$
where $r_n$ is the $n$--th {\em Birkhoff sum}. Recall that
$r_n:=r+r\circ\sigma+\cdots+r\circ\sigma^{n-1}$ for $n\geq 1$, and that there is a unique way to extend
the definition to $n\leq 0$ so that the {\em cocycle identity} $r_{m+n}=r_n+r_m\circ\sigma^n$
holds for all $m,n\in\Z$. It is given by $r_0:=0$ and $r_{n}:=-r_{|n|}\circ\sigma^{-|n|}$ for $n<0$.\label{TopMarkovFlowDefiPage}
The cocycle identity guarantees that
$\sigma_r^{\tau_1+\tau_2}=\sigma_r^{\tau_1}\circ\sigma_r^{\tau_2}$ for all $\tau_1,\tau_2\in\R$.

A TMF is topologically transitive iff its basis is a topologically transitive TMS, but the same is not true
for topological mixing. For instance, if the roof function is constant then the TMF is never
topologically mixing. By the spectral decomposition \cite{Kitchens-Book}, every TMF whose basis is a topologically transitive
TMS can be recoded as a TMF whose basis is a topologically mixing TMS. Just replace $\Sigma$ by $\Sigma_0$ and $r$ by $r_p$. \label{TMF-mixing-basis}
Let $\mu$ be a $\sigma_r$--invariant probability measure on $\Sigma_r$.

\medskip
\noindent
{\sc Induced measure:} The {\em induced measure} of $\mu$ is the unique $\sigma$--invariant probability measure $\nu$ on $\Sigma$ s.t.
$
\mu=\frac{1}{\int_\Sigma rd\nu}\int_\Sigma \int_0^{r(x)}\delta_{(x,t)}dt d\nu(x).
$

\medskip
Above, $\delta$ denotes the Dirac measure. A $\sigma_r$--invariant measure is ergodic iff its induced measure is.
Every ergodic $\sigma_r$--invariant measure on $\Sigma_r$ is carried by a
TMF whose basis is a topologically transitive TMS.

\subsection*{Bowen-Walters Metric \cite{Bowen-Walters-Metric}} This is a metric which makes  $\sigma_r:\Sigma_r\to\Sigma_r$ continuous.
Suppose first that  $r\equiv 1$ (constant suspension).

Let $\psi:\Sigma_1\to\Sigma_1$ be the suspension flow, and introduce the following terminology:
\begin{enumerate}[$\circ$]
\item {\em Horizontal segments:} Ordered pairs $[z,w]_h\in \Sigma_1\x\Sigma_1$ where  $z=(x,t)$ and $w=(y,t)$
have the same height $0\leq t<1$. The
{\em length} of a horizontal segment $[z,w]_h$ is defined as
$
\ell([z,w]_h):=(1-t)d(x,y)+td(\sigma(x),\sigma(y)).
$
\item {\em Vertical segments:} Ordered pairs $[z,w]_v\in\Sigma_1\x\Sigma_1$ where $w=\psi^t(z)$ for some $t$. The {\em length} of a vertical segment $[z,w]_v$ is
$
\ell([z,w]_v):=\min\{|t|>0:w=\psi^t(z)\}$.
\item {\em Basic paths} from $z$ to $w$: $\gamma:=(z_0=z\xrightarrow[]{t_0}z_1\xrightarrow[]{t_1}\cdots\xrightarrow[]{t_{n-2}}z_{n-1}\xrightarrow[]{t_{n-1}}z_n=w)$ with $t_i\in\{h,v\}$ such that $[z_{i-1},z_i]_{t_{i-1}}$
is a horizontal segment if $t_{i-1}=h$ , and a vertical segment if $t_{i-1}=v$. Define $\ell(\gamma):=\sum_{i=0}^{n-1} \ell([z_i,z_{i+1}]_{t_i})$.
\end{enumerate}

\medskip
\noindent
{\sc Bowen-Walters Metric on $\Sigma_1$:} $d_1(z,w):=\inf\{\ell(\gamma)\}$ where $\gamma$
ranges over all basic paths from $z$ to $w$.

\medskip
Next we consider the general case $r\not\equiv 1$. The idea is to use a canonical bijection from $\Sigma_r$ to $\Sigma_1$
and declare it to be an isometry.

\medskip
\noindent
{\sc Bowen-Walters metric on $\Sigma_r$:}
$d_r(z,w):=d_1(\vartheta_r(z),\vartheta_r(w))$, where  $\vartheta_r:\Sigma_r\to\Sigma_1$ is given by
$\vartheta_r(x,t):=(x,t/r(x))$.
\label{Bowen-Walters-Metric-Page}

\begin{lem}[\cite{Bowen-Walters-Metric,Lima-Sarig}]\label{Lemma-BW}
 $d_r$ is a metric, and $\sigma_r^t:\Sigma_r\to\Sigma_r$ is continuous  with respect to $d_r$.
 Moreover, $(t,x)\mapsto \sigma_r^t(x)$ is H\"older continuous on $[-1,1]\times\Sigma$.
\end{lem}

\subsection*{Roof functions independent of the past or future}


$r:\Sigma\to\R$ is {\em independent of the past} if  $r(x)=f(x_0,x_1,\ldots)$ for some function $f$, and
it is {\em independent of the future} if $r(x)=g(\ldots,x_{-1},x_0)$ for some function $g$ (note that we allow dependence on the zeroth coordinate).
The next lemma is an adaptation of \cite[Lemma 2]{Ratner-Flows-Bernoulli}. Let $\sigma_r:\Sigma_r\to\Sigma_r$ be a TMF
and $\mu$ be an ergodic $\sigma_r$--invariant probability measure.

\begin{lem}\label{LemmaOneSided}
$(\Sigma_r,\sigma_r,\mu)$ is isomorphic to a TMF with roof function independent of the past, and to a TMF with roof
function independent of the future. 
\end{lem}
\begin{proof}
Let us prove the first statement (the second is proved similarly).
If $\mu$ is supported on a periodic orbit, then every function is independent of the past on the support of $\mu$.
Henceforth we assume that $\mu$ does not sit on a periodic orbit.

It is well-known that there is a bounded continuous function $h^s:\Sigma\to\R$ such that
 $r^s:=r-h^s+h^s\circ\sigma$ is bounded,  H\"older continuous and independent of the past.
Proofs for $\Sigma=\Sigma(\mathfs G)$ with $\mathfs G$ finite can be found in \cite{Sinai-Gibbs,Bowen-LNM}.
As noted in \cite{Daon}, these proofs extend without much difficulty to the case where $\mathfs G$ is countable.
Since the $r^s$ produced by the proofs may take negative values, we now explain how to change $r$ and
$h^s$ to have $r^s>0$.

\medskip
\noindent
{\sc Claim:} {\em It is possible to change $r,h^s$ s.t. $0<h^s<\frac{1}{2}r$. In particular $r^s>0$.}

\medskip
\noindent
{\em Proof.\/}
Since $h^s$ is bounded, we can add a large constant to get a new $h^s$ that is positive.
The other inequality is more complicated. Let $c=\sup (h^s)<\infty$, and
take $n_0\in\N$ with $c<\frac{1}{2}n_0\inf(r)$.
Let $\nu$ be the induced measure of $\mu$. Since $\mu$ is ergodic and does not sit on a periodic orbit,
$\nu$ is non-atomic, hence there is a cylinder $[\un{b}]$ s.t. $0<\nu[\un{b}]<\frac{1}{n_0}$.
Let $ \vf_{\un{b}}(x)=\inf\{n\geq 1:\sigma^n(x)\in [\un{b}]\}$.
By the Kac formula
$\frac{1}{\nu[\un{b}]}\int_{[\un{b}]}\vf_{\un{b}} d\nu>n_0.$
Thus there exists an admissible word $\un{a}=\un{b}\,\un{\xi}\,\un{b}$ s.t. $\nu[\un{a}]>0$ and $\vf_{\un{b}}\restriction_{[\un{a}]}>n_0$.

Recode the flow using the Poincar\'e section $[\un{a}]\x\{0\}$ to obtain a suspension flow with basis $\sigma^{\vf_{\un{a}}}:[\un{a}]\to [\un{a}]$
and roof function $R=r_{\vf_{\un a}}$, where $\vf_{\un{a}}(x)=\inf\{n\geq 1:\sigma^n(x)\in [\un{a}]\}$.
The map $\sigma^{\vf_{\un{a}}}:[\un{a}]\to [\un{a}]$ admits a countable Markov partition $$S:=\{[\un{a},\un{\xi},\un{a}]:\vf_{\un{a}}\upharpoonright_{[\un{a},\un{\xi},\un{a}]}=|\un{a}|+|\un{\xi}|\}\setminus\{\emptyset\}.$$ Coding with $S$, $\sigma^{\vf_{\un{a}}}:[\un{a}]\to [\un{a}]$ becomes a TMS, therefore the suspension flow is a TMF.
Under this new coding, $R^s:=R-h^s+h^s\circ\sigma^{\vf_{\un{a}}}$ is independent of the past and
H\"older continuous.
Note that $\vf_{\un{a}}\geq \vf_{\un{b}}>n_0\Rightarrow\inf R>n_0\inf(r)>2c\Rightarrow h^s<\frac{1}{2}R$.

\medskip
Henceforth we assume, without loss of generality, that $0<h^s<\frac{1}{2}r$ for the original flow. Then $r^s$ is bounded,  positive and {\em uniformly bounded away from zero}.
This allows us form the TMF $\sigma_{r^s}:\Sigma_{r^s}\to \Sigma_{r^s}$.
This TMF is isomorphic to $\sigma_r:\Sigma_r\to\Sigma_r$ via the conjugacy
$$
\vartheta_s(x,\xi)=\begin{cases}
(x,\xi-h^s(x)) & \text{, if }\xi\geq h^s(x)\\
(\sigma^{-1}(x),\xi+r(\sigma^{-1}(x))-h^s(\sigma^{-1}(x))) & \text{, if }0\leq \xi<h^s(x),
\end{cases}
$$
which  recodes $\Sigma_r$ using the Poincar\'e section
$
\{(x,h^s(x)):x\in\Sigma\}$.
\end{proof}

\subsection*{Strong manifolds and the Bowen-Marcus Cocycles \cite{Bowen-Marcus}}
 The   {\em strong stable} and {\em strong unstable manifolds of $(x,t)$} are:
\begin{enumerate}[$\circ$]\label{Page-strong-manifold}
\item $W^{ss}(x,t):=\{(y,s):d_r(\sigma_r^\tau(x,t),\sigma_r^\tau(y,s))\xrightarrow[\tau\to\infty]{}0\}$.
\item $W^{su}(x,t):=\{(y,s):d_r(\sigma_r^{-\tau}(x,t),\sigma_r^{-\tau}(y,s))\xrightarrow[\tau\to\infty]{}0\}$.
\end{enumerate}
These are not manifolds, but they play the same role as their smooth analogues in hyperbolic dynamics.

To calculate $W^{ss}, W^{su}$  we make the following definitions. Assume $x$ is not pre-periodic (i.e. there are no $m,n$ s.t. $x_m^\infty$ or $x_{-\infty}^n$ is a periodic sequence).
Let $W^{ws}(x):=\{y\in\Sigma:\exists m,n\textrm{ s.t. }y_m^\infty=x_n^\infty\}$ and define
$P^s(x,\cdot):W^{ws}(x)\to\R$ by
$
P^s(x,y):=\lim\limits_{k\to\infty} [r_{m+k}(y)-r_{n+k}(x)]\textrm{ for some (every)  $m,n$ s.t. $y_m^\infty=x_n^\infty$}
$.\label{Page-Weak-(un)stable}
Similarly,
let $W^{wu}(x):=\{y\in\Sigma:\!\exists m,n\textrm{ s.t. }y_{-\infty}^{m}=x_{-\infty}^{n}\}$, and set
$P^u(x,\cdot)\!:\!W^{wu}(x)\to\R$ by
$
P^u(x,y):=\lim\limits_{k\to-\infty} [r_{m+k}(y)-r_{n+k}(x)]\textrm{ for some (every) $m,n$ s.t. $y_{-\infty}^m=x_{-\infty}^n$}
$.

These definitions are independent of the choice of $m,n$, because in the non-pre-periodic case any two possible pairs $(m,n), (m',n')$ satisfy $m'=m+k_0, n'=n+k_0$ for some $k_0\in\Z$.
The limits which define $P^\tau(\cdot,\cdot)$ exist because they are the limits of the partial sums of the  series
$r_m(y)-r_n(x)+\sum_{k=0}^\infty[r(\sigma^{m+ k}(y))-r(\sigma^{n+ k}(x))]$ $(\tau=s)$ or
$r_m(y)-r_n(x)-\sum_{k=1}^\infty [r(\sigma^{m-k}(y))-r(\sigma^{n-k}(x))]$ $(\tau=u)$.
Since $r$ is H\"older continuous, the summands decay exponentially fast, and these series converge.

\begin{lem}[\cite{Bowen-Marcus}]\label{Lemma-P}
Suppose $x$ is not pre-periodic, then for  $\tau=s,u$ it holds:
\begin{enumerate}[$(1)$]
\item {\sc Bowen-Marcus condition:} $(y,s)\in W^{s\tau}(x,t)$ iff $y\in W^{w\tau}(x)$ and
$s-t=P^\tau(x,y)$.
\item {\sc Shift identity:} $P^\tau(\sigma x,\sigma y)-P^\tau(x,y)=r(x)-r(y)$  wherever defined.
\item {\sc Cocycle equation:} For all  $y,z\in W^{w\tau}(x)$,
$
P^\tau(x,y)+P^\tau(y,z)=P^\tau(x,z).
$
In particular, $P^\tau(x,x)=0$ and $P^\tau(x,y)=-P^\tau(y,x)$.
\item {\sc H\"older property:} There are $C>0$, $0<\alpha<1$ s.t. $|P^\tau(x,y)|\leq C d(x,y)^\alpha$
for all $y\in W^\tau_{\loc}(x):=\{y\in\Sigma:y_0^\infty=x_0^\infty\}$.
\end{enumerate}
\end{lem}
\noindent
$P^s(\cdot,\cdot), P^u(\cdot,\cdot)$ are called the {\em Bowen-Marcus cocycles}.

\section{Equilibrium measures for Topological Markov Flows}
\subsection*{Equilibrium measures}
Let $\sigma_r:\Sigma_r\to\Sigma_r$ be a TMF, and let $\Phi:\Sigma_r\to\R$ be bounded and continuous.
The {\em (variational) topological pressure} of $\Phi$ is
$$
P_{\rm top}(\Phi):=\sup\left\{h_\mu(\sigma_r^1)+\int \Phi d\mu: \mu\text{ is $\sigma_r$--invariant Borel probability measure}\right\}.
$$

\medskip
\noindent
{\sc Equilibrium measure:} $\mu$ is called an {\em equilibrium measure} (for the potential $\Phi$
and the flow $\{\sigma_r\}$) if $h_\mu(\sigma_r^1)+\int \Phi d\mu=P_{\rm top}(\Phi)$.

\medskip
In this section, we will describe the equilibrium measures when $\Sigma$ is topologically mixing,
$\Phi$ is bounded and H\"older continuous, and $P_{\rm top}(\Phi)<\infty$.
Instead of describing them directly, we describe the {one-sided version} of their {induced measures}.
Let $\mu$ be a $\sigma_r$--invariant probability measure, and let $\nu$ be its induced measure.
$\nu$ is a $\sigma$--invariant probability measure on $\Sigma$.

\medskip
\noindent
{\sc One-sided TMS:} Let $\pi_s:x\in\Sigma\mapsto (x_0,x_1,\ldots)$. The {\em one-sided TMS}
is the discrete-time topological dynamical system $\sigma_s:\Sigma^s\to\Sigma^s$ where
$$
\Sigma^s=\{\pi_s(x):x\in\Sigma\}
$$
and $\sigma_s:\{x_i\}_{i\geq 0}\mapsto \{x_{i+1}\}_{i\geq 0}$ is the {\em one-sided left shift}.

\medskip
\noindent
{\sc One-sided version of $\nu$:}
The {\em one-sided version} of $\nu$ is the probability measure $\nu^s:=\nu\circ\pi_s^{-1}$.
$\nu^s$ is a $\sigma_s$--invariant probability measure on $\Sigma^s$.


\medskip
$\nu^s$ determines $\nu$ since $\nu\circ\sigma^{-1}=\nu$, and $\nu$ determines $\mu$.
Here is the description of $\nu^s$.

\begin{thm}\label{Thm-RPF} Let $\sigma_r:\Sigma_r\to\Sigma_r$ be a topologically transitive TMF and
 $\Phi:\Sigma_r\to\R$ be bounded and H\"older continuous with $P_{\rm top}(\Phi)<\infty$.
Let $\mu$ be an equilibrium measure for $\Phi$, and $\nu$ its induced measure.
Then the one-sided version of $\nu$ has the form $\nu^s=h^s \xi^s$, where:
\begin{enumerate}[$(1)$]
\item $h^s$ is a positive function on $\Sigma^s$, and $\xi^s$ is a positive measure on $\Sigma^s$.
\item There is $\phi^s:\Sigma^s\to\R$ bounded H\"older continuous with $P_{\rm top}(\phi^s)<\infty$
s.t. $Lh^s=\lambda h^s$ and $L^\ast\xi^s=\lambda \xi^s$, where $\lambda=\exp[P_{\rm top}(\phi^s)]$ and
$L$ is the Ruelle operator of $\phi^s$,
$(Lf)(x_0^\infty)=\sum_{\sigma_s(y_0^\infty)=x_0^\infty}\exp[{\phi^s(y_0^\infty)}]f(y_0^\infty)$ for all $f:\Sigma^s\to\R$.
\item $h^s(x)=\lim\limits_{n\to\infty}\frac{1}{\xi^s[\un{a}]}\l^{-n}(L^n 1_{[\un{a}]})(x)$ for every cylinder  $[\un{a}]$ and  $x\in\Sigma^s$. 
\item $\log h^s$ is uniformly H\"older continuous on cylinders of length one at the zeroth position.
\item $\nu^s$ is ergodic.
\end{enumerate}
\end{thm}

\begin{proof}
Bowen and Ruelle proved the theorem in  \cite{Bowen-Ruelle-SRB} for TMF built from finite graphs, using Ruelle's Perron-Frobenius Theorem \cite{Lanford-Ruelle,Bowen-LNM,Ruelle-TDF-book}. Since Ruelle's Perron-Frobenius Theorem
is false for general infinite graphs, we sketch the modifications needed to treat our case.

\medskip
\noindent
{\sc Claim 1.\/} {\em $\nu$ is an equilibrium measure for
$
\phi(x):=\int_0^{r(x)}\Phi(x,t)dt-P_{\rm top}(\Phi)r(x).
$
$\phi:\Sigma\to\R$ is bounded H\"older continuous with $P_{\rm top}(\phi)=0$.}

\medskip
\noindent
{\em Proof.\/} This is proved exactly as in \cite{Bowen-Ruelle-SRB}. $\phi$ is clearly bounded
H\"older continuous. By the Abramov entropy formula \cite{Abramov-Entropy-Flows},
$h_{\mu}(\sigma_r)=\frac{1}{\int r d\nu}h_{\nu}(\sigma)$. Hence
$\frac{h_{\nu}(\sigma)+\int_\Sigma\int_0^{r(x)}\Phi(x,t)dtd\nu(x)}{\int r d\nu}\leq P_{\rm top}(\Phi)$,
with equality iff $\mu$ is an equilibrium measure for $\Phi$. This can be rewritten as
$h_{\nu}(\sigma)+\int_\Sigma \phi(x)d\nu(x)\leq 0$,
with equality iff $\nu$ is an equilibrium measure for $\phi$. Therefore $P_{\rm top}(\phi)=0$, and
$\mu$ is an equilibrium measure for $\Phi$ iff its induced measure $\nu$ is an equilibrium measure for $\phi$.

\medskip
\noindent
{\sc Claim 2.\/} {\em $\nu$ is an equilibrium measure for a bounded H\"older continuous
potential that is independent of the past and has zero pressure.}

\medskip
\noindent
{\em Proof.\/}  By \cite{Sinai-Gibbs,Bowen-LNM,Daon} there is a  bounded H\"older continuous function
$v:\Sigma\to\R$ s.t.
$
\phi+v-v\circ\sigma
$
is independent of the past. Since $\int (v-v\circ\sigma)dm=0$ for every $\sigma$--invariant probability
measure $m$, $P_{\rm top}(\phi+v-v\circ\sigma)=P_{\rm top}(\phi)=0$.

\medskip
Now we proceed to the proof of Theorem \ref{Thm-RPF}. By claims 1--2, there is $\phi^s:\Sigma^s\to\R$
bounded H\"older continuous s.t. $\phi^s\circ\pi_s=\phi+v-v\circ\sigma$, $\nu$ is an equilibrium
measure for $\phi^s\circ\pi_s$, and $P_{\rm top}(\phi^s\circ\pi_s)=0$. We want to conclude that
$\nu^s$ is an equilibrium measure for $\phi^s$, and that $P_{\rm top}(\phi^s)=0$.

If $\nu$ is a $\sigma$--invariant probability measure then $(\Sigma,\nu,\sigma)$ is the natural extension of
$(\Sigma^s,\nu^s,\sigma_s)$. Conversely, if $\nu^s$ is a $\sigma_s$--invariant probability measure
then it is the one-sided version of some $\sigma$--invariant probability measure $\nu$ (its natural extension).
Since natural extensions preserve entropy, $P_{\rm top}(\phi^s)=P_{\rm top}(\phi^s\circ\pi_s)=0$, and
$\nu$ is an equilibrium measure for $\phi^s\circ\pi_s$ iff $\nu^s$ is an equilibrium measure for $\phi^s$.

The structure of equilibrium measures for H\"older continuous potentials on one-sided TMS was determined in \cite{Buzzi-Sarig}. There it is shown that if $\Sigma^s$ is topologically mixing (a consequence of the topological mixing of $\Sigma$),
then $\phi^s$ is positive recurrent in the sense of \cite{Sarig-Null-Recurrent}, and parts (1)--(3) of the theorem hold.
Also, if the equilibrium measure exists then it is unique \cite[Thm 1.1]{Buzzi-Sarig}, and this gives part (5).
Part (4) follows from part (3) and the boundedness and H\"older continuity of $\phi^s$.
\end{proof}

\begin{cor}\label{Cor-Eq-are-Ergodic}
Suppose $\Sigma_r$ is a topologically transitive TMF, and $\Phi$ is a bounded H\"older continuous potential with finite pressure. Then $\Phi$ has at most one equilibrium measure and if this measure exists then it is ergodic.
\end{cor}
\begin{proof}
By Theorem \ref{Thm-RPF} $\nu^s$ is ergodic. Therefore its natural extension $\nu$ is ergodic. If the induced measure is ergodic, then the original measure is ergodic. It follows that every equilibrium measure is ergodic. This implies that the equilibrium measure is unique: if there were two equilibrium measures, then  their average would have been a non-ergodic equilibrium measure.
\end{proof}

\subsection*{Conditional measures of the induced measure} Theorem \ref{Thm-RPF} can be used to construct  the conditional measures  $\nu(\cdot|x_0^\infty)$ for all, rather than almost all, $x\in\Sigma^s$.
The basic tool is the   {\em $g$--function} of $\nu$. This is the function $g:\Sigma^s\to\R$ given by
$$
g:=\frac{e^{\phi^s} h^s}{\lambda h^s\circ\sigma}=\frac{d\nu^s}{d(\nu^s\circ\sigma_s)}\cdot
$$
The reader can check that $g>0$ and $\sum_{\sigma_s(y_0^\infty)=x_0^\infty}g(y_0^\infty)=1$, whence  $0<g\leq 1$. Thus $g$ is a $g$--function in the sense of \cite{Keane-g-measures}. The function $\log g$ is bounded and uniformly H\"older continuous on cylinders of length two, since $\exp \phi^s$, $\log h^s$ are bounded and uniformly  H\"older continuous on
cylinders of length one at the zeroth position.

\begin{thm}[\cite{Ledrappier-Principe-Variationnel}] Let $\nu$, $\nu^s$, $L$ as in Theorem \ref{Thm-RPF}.
\begin{enumerate}[$(1)$]
\item If $f\in L^1(\nu^s)$ then $\E_{\nu^s}(f|x_1^\infty)=\sum_{\sigma_s(y_0^\infty)=x_1^\infty}g(y_0^\infty)f(y_0^\infty)$ $\nu^s$--a.e.
\item ${\nu^s}(x_0|x_{1},x_{2},\ldots):=\lim\limits_{k\to\infty}\nu([x_0]|{_1[}x_{1},\ldots,x_{k}])=g(x_0^\infty)$ $\nu^s$--a.e.
\item $\lim\limits_{k\to\infty}\nu({_{-n}[}x_{-n},\ldots,x_{-1}]|{_0[}x_{0},\ldots,x_{k}])$ is equal $\nu$--a.e. to
 \begin{equation}\label{g-function}
\nu(x_{-n},\ldots,x_{-1}|x_{0}^{\infty}):=g_n(x_{-n}^\infty):=g(x_{-n}^\infty) g(x_{-n+1}^\infty)\cdots g(x_{-1}^\infty).
\end{equation}
\end{enumerate}
\end{thm}
\begin{proof}
Part (1) follows from the equations  $\nu^s=h^s\xi^s$, $L h^s=\lambda h^s$, $L^\ast\xi^s=\lambda\xi^s$ as in \cite{Ledrappier-Principe-Variationnel}. Part (2) follows from part (1) and the martingale convergence theorem. Part (3) follows from part (2) and the invariance of $\nu$.  
\end{proof}
\noindent

One should view (\ref{g-function}) as a consistent set of equations which determine the conditional probability measure $\nu(\cdot|x_0^\infty)$ on $W^s_{\rm loc}(x)$, by specifying
the weights these measures give to cylinders. By consistent we mean that $\sum_{\sigma_s(y_0^\infty)=x_0^\infty}g(y_0^\infty)=1$.
Henceforth, we define $\nu(\cdot|x_0^\infty)$ as follows.

\medskip
\noindent
{\sc Measure $\nu(\cdot|x_0^\infty)$:} $\nu(\cdot|x_0^\infty)$ is the unique probability measure on $W^s_{\rm loc}(x)$
s.t.
$
\nu(\un{a}|x_0^\infty):=g_n(\un{a}x_0^\infty)$ for admissible words $\un{a}$.

\begin{lem}\label{Lemma-Non-Atomic}
Let $\nu$ be as in Theorem \ref{Thm-RPF}. If $\Sigma_r$ is topologically transitive and $\Sigma_r$ is not a union of cycles,  then  $\nu(\cdot|x_0^\infty)$ is non-atomic for $\nu$--a.e. $x\in\Sigma$.
\end{lem}

\begin{proof}
Since $\Sigma_r$ is topologically transitive and $\Sigma_r$ is not a union of cycles, the same is true for $\Sigma$.
In particular there is a state $b$ with in-degree at least two. Fix one such edge $a\to b$.
Since $\sum_{\sigma_s(z_0^\infty)=y_0^\infty}g(z_0^\infty)=1$,
we have $g(z)<1$ for every $z\in\Sigma^s$ s.t. $(z_0,z_1)=(a,b)$. By the H\"older continuity of $\log g$,
we can find a word $\un{w}:=(a,b,b_2,\ldots,b_n)$ s.t.  $g\restriction_{[\un{w}]}<1$.
By (\ref{g-function}), $\nu(\{z\}|x_0^\infty)=0$ whenever
$
z\in Z:=\{z\in\Sigma: z_{n}^{n+|\un{w}|-1}=\un{w}\text{ for infinitely many }n<0\}.
$
The conclusion is that $\nu(\cdot|x_0^\infty)$ is non-atomic for every $x_0^\infty$ s.t. $\nu(Z|x_0^\infty)=1$.

Let us show that this last condition is true $\nu$--a.e. By Theorem \ref{Thm-RPF},
$\nu$ is ergodic and positive on cylinders, hence $\nu(Z)=1$, i.e.
$\int \nu(Z|x_0^\infty)d\nu(x)=\nu(Z)=1$, so $\nu(Z|x_0^\infty)=1$ for $\nu$--a.e. $x\in\Sigma$.
\end{proof}

\subsection*{Local product structure of the induced measure}
Let $\sigma:\Sigma\to\Sigma$ be a TMS. The following definitions are motivated by smooth ergodic theory, see e.g.  \cite{Barreira-Pesin-Non-Uniform-Hyperbolicity-Book}:

\begin{enumerate}[$\circ$]
\item $W^s(x):=\{y\in\Sigma: d(\sigma^n(x),\sigma^n(y))\xrightarrow[n\to\infty]{}0\}
=\{y\in\Sigma:\exists n\text{ s.t. } y_n^\infty=x_n^\infty\}$.
\item $W^u(x):=\{y\in\Sigma: d(\sigma^{n}(x),\sigma^{n}(y))\xrightarrow[n\to-\infty]{}0\}
\!=\!\{y\in\Sigma:\exists n\text{ s.t. } y_{-\infty}^n=x_{-\infty}^n\}$.
\item $W^s_{\loc}(x):=\{y\in\Sigma:y_0^\infty=x_0^\infty\}$.
\item $W^u_{\loc}(x):=\{y\in\Sigma:y_{-\infty}^0=x_{-\infty}^0\}$.
\end{enumerate}

\medskip
\noindent
{\sc Smale bracket of points:} Let $x,y\in\Sigma$ with $x_0=y_0$.
The {\em Smale bracket} of $x,y$ is
$[x,y]:=z$ where $z_i=x_i$ for $i\leq 0$ and $z_i=y_i$ for $i\geq 0$.

\medskip
If $x_0=y_0=v$, then $[W^s_{\loc}(x),W^u_{\loc}(y)]=\{[x',y']:x'\in W^s_{\loc}(x),y'\in W^s_{\loc}(y)\}=[v]=\{z\in\Sigma:z_0=v\}$.
We can also consider the Smale products of measures.
Let $\alpha^s_{x}$, $\beta^u_{y}$ be finite measures on $W^s_{\loc}(x)$, $W^u_{\loc}(y)$ respectively.

\medskip
\noindent
{\sc Smale bracket of measures:} The {\em Smale bracket} of $\alpha^s_{x},\beta^u_{y}$ is
a finite measure on $[W^s_{\loc}(x),W^u_{\loc}(y)]={[v]}$ defined by
$$
(\alpha^s_{x}\star\beta^u_{y})(E):=\int_{W^u_{\loc}(y)}\int_{W^s_{\loc}(x)} 1_E([x',y'])d\alpha^s_{x}(x')d\beta^u_{y}(y'),\ \ (E\textrm{ Borel measurable}).
$$

The Smale product produces measures on $\Sigma$ out of measures on $W^s_{\loc}(x),W^u_{\loc}(y)$.
We can also produce measures on  $W^s_{\loc}(x),W^u_{\loc}(y)$ from measures on $\Sigma$. Let:
\begin{enumerate}[$\circ$]
\item $p^s_{x}:[x_0]\to W^s_{\loc}(x)$, $p^s_{x}(\cdot)=[\cdot,x]$.
\item $p^u_{x}:[x_0]\to W^u_{\loc}(x)$, $p^u_{x}(\cdot)=[x,\cdot]$.
\end{enumerate}

\medskip
\noindent
{\sc Projection measures:} The {\em projections} of $\nu$ on $W^s_{\loc}(x),W^u_{\loc}(y)$ are
\begin{equation}\label{nu-s}
\left\{
\begin{aligned}
\nu^s_{x}&:=\nu\circ (p^s_{x})^{-1},\ \ \textrm{a measure on $W^s_{\loc}(x)$},\\
\nu^u_{y}&:=\nu\circ (p^u_{y})^{-1}, \ \ \textrm{a measure on $W^u_{\loc}(y)$}.
\end{aligned}\right.
\end{equation}


\medskip
Note that for $\tau=u,s$:
\begin{enumerate}[$\circ$]
\item $\nu^\tau_{x}=\nu^\tau_{y}$ iff $W^\tau_{\loc}(x)=W^\tau_{\loc}(y)$.
\item $\nu^\tau_{x}=(\nu^\tau_y\circ p^\tau_{y})\upharpoonright_{W^\tau_{\loc}(x)}$ whenever $x_0=y_0$.
\end{enumerate}

\medskip
\noindent
{\sc Local product structure:} $\nu$ is said to have {\em local product structure} if for every
$x,y\in\Sigma$ s.t. $x_0=y_0=v$ we have
$\nu^s_{x}\star\nu^u_{y}\sim\nu\upharpoonright_{ [v]}$.


\begin{thm}\label{Prop-Prod-Strctr-measure}
Let $\mu$ be an equilibrium measure of a bounded H\"older continuous potential with finite pressure on a topologically transitive TMF, and let $\nu$ be its induced measure. Then $\nu$ is globally supported, and $\nu$ has local product structure.
\end{thm}
\begin{proof}

Let $\sigma: \Sigma\to\Sigma$ be the associated TMS, and let $\mathfs G$ be a directed graph associated to
$\Sigma$. Since the TMF is topologically transitive, $\sigma:\Sigma\to\Sigma$ is topologically transitive,
hence any two vertices on $\mathfs G$ can be joined  by a path.

\medskip
\noindent
{\sc Claim:\/} {\em Every non-empty cylinder on $\Sigma$ has positive $\nu$--measure, and for every edge
$v\to w$ there is a constant $C_{vw}>1$ s.t. if $m<0, n>0$ and ${_m[}v_m,\ldots,v_n]\neq\emptyset$, then
$
C_{v_0 v_1}^{-1}\leq \frac{\nu({_{m}[}v_m,\ldots,v_n])}{\nu({_{m}[}v_m,\ldots,v_{0}])\nu({_0[}v_0,\ldots,v_n])}\leq C_{v_0 v_1}.
$}

\medskip
\noindent
{\em Proof of the claim.\/} Let  $\nu^s$ be the one-sided version of $\nu$.  Theorem \ref{Thm-RPF} implies, as in
\cite[Corollary 3.2]{Sarig-Bernoulli-JMD},
the existence of constants  $K_v,D_{vw}>0$ s.t.
\begin{enumerate}[(a)]
\item
$
K_{a_{n-1}}^{-1}\leq \frac{\nu^s({[}a_0,\ldots,a_{n-1},b_0,\ldots,b_{k-1}])}{\nu^s({[}a_0,\ldots,a_{n-1}])\nu^s({[}b_0,\ldots,b_{k-1}])}\leq K_{a_{n-1}}
$ for all $\un{a},\un{b}$ s.t. $[\un{a},\un{b}]\neq \emptyset$,
\item $D_{a_{n-1}b_0}^{-1}\leq \frac{\nu^s({[}b_0,\ldots,b_{k-1}])}{\nu^s({[}a_{n-1},b_0,\ldots,b_{k-1}])}\leq D_{a_{n-1}b_0}$ whenever $[a_{n-1},\un{b}]\neq \emptyset$.
\end{enumerate}
By (a)--(b), there are constants $C_{vw}$ s.t. for all $\un{a},\un{b}$ with $[\un{a},\un{b}]\neq \emptyset$ we have:
$$
C_{a_{n-1}b_0}^{-1}\leq \frac{\nu^s({[}a_0,\ldots,a_{n-1},b_0,\ldots,b_{k-1}])}{\nu^s({[}a_0,\ldots,a_{n-1}])\nu^s({[}a_{n-1},b_0,\ldots,b_{k-1}])}\leq C_{a_{n-1}b_0}.
$$
Substituting  $\un{a}=(v_m,\ldots,v_0),\un{b}=(v_1,\ldots,v_n)$ gives the  claim.

\medskip
By the claim, if $E$ is a cylinder contained in $ [v,w]$ and $x,y\in {[}v]$ then:
\begin{equation}\label{product-identity}
C_{vw}^{-1}\times(\nu^s_{x}\star\nu^u_{y})(E)\leq \nu(E)\leq C_{vw}\times (\nu^s_{x}\star\nu^u_{y})(E).
\end{equation}
The collection of cylinders $E\subset {[}v,w]$ satisfying (\ref{product-identity})  is closed under increasing unions and decreasing intersections. By the monotone class theorem, (\ref{product-identity})  holds for every Borel set $E\subset {[}v,w]$,
whence $\nu^s_{x}\star\nu^u_{y}\sim\nu\upharpoonright_{{[v]}}$.
\end{proof}

\begin{cor}\label{Cor-Prod-Structure}
Let $\nu$ be as in the previous theorem. If $E\subset\Sigma$ is Borel and $\nu(E)=0$, then
$\nu^s_{x}(E)=\nu^u_{x}(E)=0$ for $\nu$--a.e. $x$.
\end{cor}

\begin{proof}
Let $\Omega_v:=\{x\in\Sigma:x_0=v\textrm{ and }\nu^s_{x}(E)>0\}$, and  assume by  contradiction
that $\nu(\Omega_v)>0$ for some $v$. Since $\nu$ has local product structure, if $x,y\in {[}v]$ then:
$$
\int_{W^u_{\loc}(y)}\int_{W^s_{\loc}(x)}1_{\Omega_v}([x',y'])d\nu^s_{x}(x')d\nu^u_{y}(y')>0.
$$
Note that $[x',y']\in\Omega_v\Leftrightarrow \nu^s_{[x',y']}(E)>0\overset{!}{\Leftrightarrow}\nu^s_{y'}(E)>0
\Leftrightarrow y'\in \Omega_v$
($\overset{!}{\Leftrightarrow}$ is because $\nu^s_{[x',y']}=\nu^s_{y'}$). Hence $1_{\Omega_v}([x',y'])=1_{\Omega_v}(y')$. Calculating the double integral, we find that
$\nu^u_y[\Omega_v]\nu^s_{x}[W^s_{\loc}(x)]>0\Rightarrow\nu^u_{y}[\Omega_v]>0$. We use this to get a contradiction.

\medskip
Let $y'\in\Omega_v$. Using that $\nu^s_{y'}=(\nu^s_{x}\circ p^s_{x})\upharpoonright_{W^s_{\loc}(y')}$, we have
\begin{align*}
0&<\nu^s_{y'}(E)=(\nu^s_{x}\circ p^s_{x})[E\cap W^s_{\loc}(y')]=\nu^s_{x}\{x'\in W^s_{\loc}(x):[x',y']\in E\}\\
&=\int_{W^s_{\loc}(x)}1_E([x',y'])d\nu^s_{x}(x').
\end{align*}
Since $\nu^u_{y'}[\Omega_v]>0$, if we integrate this inequality we obtain
$$
\int_{W^u_{\loc}(y)}\left(\int_{W^s_{\loc}(x)}1_{E}([x',y'])d\nu^s_{x}(x')\right)d\nu^u_{y}(y')>0,
$$
thus $(\nu^s_{x}\star\nu^u_{y})(E)>0$. Since $\nu$ has local product structure, this gives that $\nu(E)>0$, a contradiction.
We have just proved that $\nu[\Omega_v]=0$ for every vertex $v$, whence $\nu^s_{x}(E)=0$ for $\nu$--a.e. $x$.
By symmetry, $\nu^u_{x}(E)=0$ for $\nu$--a.e. $x$.
\end{proof}

\section{The Pinsker factor of a topological Markov flow}\label{section-Pinsker}

\subsection*{Review of general theory}

Let $(X,\mathfs B,\mu,T)$ be an {\em automorphism}, i.e.
$(X,\mathfs B,\mu)$ is a non-atomic Lebesgue probability space and $T$ is an invertible
transformation preserving $\mu$. Given $E\in\mathfs B$, let $\alpha_E=\{E,X\setminus E\}$.

\medskip
\noindent
{\sc Pinsker factor:} $E\in\mathfs B$ is called a {\em Pinsker set} if $h_\mu(T,\alpha_E)=0$.
The {\em Pinsker $\sigma$--algebra} is $\mathfs P(T):=\{E\in \mathfs B: E\text{ is a Pinsker set}\}$.
$(X,\mathfs P(T),\mu,T)$ is called the {\em Pinsker factor} of $(X,\mathfs B,\mu,T)$.

\medskip
$\mathfs P(T)$ is a $T$--invariant $\sigma$-algebra \cite{Pinsker}, hence $(X,\mathfs P(T),\mu,T)$ is
indeed a factor. $(X,\mathfs P(T),\mu,T)$ has zero entropy,
and if $\mathfs A\subset\mathfs B$ s.t. $(X,\mathfs A,\mu,T)$ is a factor of zero entropy then
$\mathfs A\subset\mathfs P(T)$ modulo $\mu$. Therefore $(X,\mathfs P(T),\mu,T)$ is the largest factor
of $(X,\mathfs B,\mu,T)$ with zero entropy.

\medskip
\noindent
{\sc Completely positive entropy:} $(X,\mathfs B,\mu,T)$ is said to have {\em completely positive entropy}
if it has a trivial Pinsker factor, i.e. if $\mathfs P(T)=\{\emptyset,X\}$ modulo $\mu$.

\medskip
Note that $(X,\mathfs B,\mu,T)$ has completely positive entropy iff all of its non-trivial factors have positive entropy.

\medskip
\noindent
{\sc Tail $\sigma$--algebra:} Given a $\sigma$--algebra $\mathfs A\subset\mathfs B$ with
$T^{-1}\mathfs A\subset\mathfs A$, the {\em tail $\sigma$--algebra}
of $\mathfs A$ is ${\rm Tail}(\mathfs A):=\bigcap_{n\geq 0} T^{-n}\mathfs A$.

%


\medskip
\noindent
{\sc K property:} $(X,\mathfs B,\mu,T)$ has the {\em K property} if there is a $\sigma$--algebra
$\mathfs A\subset\mathfs B$ s.t.
\begin{enumerate}[(a)]
\item $T^{-1}\mathfs A\subset  \mathfs A$,
\item $\bigvee\limits_{i=0}^\infty T^{i}\mathfs A=\mathfs B$ modulo $\mu$,
\item ${\rm Tail}(\mathfs A)=\{\emptyset,X\}$ modulo $\mu$.
\end{enumerate}

\begin{thm}[Rokhlin \& Sinai \cite{Rohlin-Sinai}]
$(X,\mathfs B,\mu,T)$ has the K property iff it has completely positive entropy.
\end{thm}

The K property is stronger than mixing. It implies continuous Lebesgue spectrum \cite{Rohlin-Exactness}, and the mixing property below,
called {\em K-mixing}, see \cite[\S10.8]{Cornfeld-Fomin-Sinai}. Write $\delta$--a.e. when a property holds for a set of atoms with total measure
$\geq 1-\delta$.

\begin{thm}\label{Thm-K-Mixing}
Let $(X,\mathfs B,\mu,T)$ be an automorphism with the K property, $B\in\mathfs B$, and
$\beta$ a finite measurable partition of $X$. Then for every $\delta>0$ there is $N_0=N_0(B,\delta)$
s.t. for all $N'>N\geq N_0$ and $\delta$--a.e. $A\in\bigvee_{k=N}^{N'}T^k\beta$ it holds
$\left|\mu(B|A)-\mu(B)\right|<\delta$.
\end{thm}


\medskip
Now let $\textsf{T}=(X,\mathfs B,\mu,\{T^t\})$ be a flow. It is known that $h_\mu(T^t)=|t|h_\mu(T^1)$ and
$\mathfs P(T^t)=\mathfs P(T^1)$, $\forall t\neq 0$ \cite{Abramov-Entropy-Flows,Gurevich-K-flows}.
The {\em Pinsker $\sigma$--algebra} of $\textsf T$ is defined as $\mathfs P(T^1)$.
$\textsf T$ is said to have {\em completely positive entropy} if its Pinsker factor is trivial iff
$\exists t\neq 0$ s.t. $(X,\mathfs B,\mu,T^t)$ is an automorphism with completely positive entropy.
$\textsf T$ is said to have the {\em K property} if $(X,\mathfs B,\mu,T^1)$ is an automorphism with
the K property iff $\exists t\neq 0$ s.t. $(X,\mathfs B,\mu,T^t)$ is an automorphism with the K property.
$\textsf T$ has the K property iff it has completely positive entropy, and it implies K-mixing \cite{Cornfeld-Fomin-Sinai}.
The next theorem is a tool for proving the K property. Given a $\sigma$--algebra $\mathfs A$
with $T^{-t}\mathfs A\subset\mathfs A$, $\forall t>0$,
let ${\rm Tail}(\mathfs A):=\bigcap\limits_{t>0} T^{-t}\mathfs A$ be the {\em tail $\sigma$--algebra}
of $\mathfs A$.

\begin{thm}[Rokhlin \& Sinai \cite{Rohlin-Sinai}]\label{Thm-Rohlin-Sinai}

Let $\textsf{T}=(X,\mathfs B,\mu,\{T^t\})$ be a flow, and let $\mathfs A\subset\mathfs B$ be a
$\sigma$--algebra s.t.
\begin{enumerate}[{\rm (a)}]
\item $T^{-t}\mathfs A\subset  \mathfs A$, $\forall t>0$,
\item $\bigvee\limits_{t>0} T^{t}\mathfs A=\mathfs B$ modulo $\mu$.
\end{enumerate}
Then $\mathfs P(T)\subset {\rm Tail}(\mathfs A)$ modulo $\mu$.
\end{thm}



\subsection*{An upper bound for the Pinsker factor of a TMF }
We now construct $\sigma$-algebras as in Theorem \ref{Thm-Rohlin-Sinai} for a topologically transitive TMF.
The construction follows \cite{Gurevich-K-flows,Ratner-Flows-K}.

Let $\sigma_r:\Sigma_r\to\Sigma_r$ be a topologically transitive TMF. By Lemma \ref{LemmaOneSided},
$\sigma_r:\Sigma_r\to\Sigma_r$ is isomorphic to a TMF $\sigma_{r^s}:\Sigma_{r^s}\to\Sigma_{r^s}$
s.t. $r^s$  is independent of the past. Let  $\vartheta_s:\Sigma_r\to\Sigma_{r^s}$ be the isomorphism,
$\vartheta_s\circ\sigma_r^t=\sigma_{r^s}^t\circ\vartheta_s$,  $\forall t\in\R$. Points in $\Sigma_{r^s}$ will
be decorated by over bars as in $(\ov{x},\ov{\xi})$.

Given $(x,\xi)\in\Sigma_r$, let $(\ov{x},\ov{\xi}):=\vartheta_s(x,\xi)$ and define
$$
W^{ss}_{\loc}(x,\xi):=\vartheta_s^{-1}\{(\ov{y},\ov{\xi})\in\Sigma_{r^s}:\ov{y}_0^\infty=\ov{x}_0^\infty\}.
$$
Any two such sets are either equal or disjoint, hence $\{W^{ss}_{\loc}(x,\xi)\}$ is a partition of $\Sigma_r$.
Let $\mathfs W^{ss}_{\loc}$ be the $\sigma$--algebra generated by $\{W^{ss}_{\loc}(x,\xi)\}$.
$\mathfs W^{ss}_{\loc}$ is generated by the countable collection
of sets
$
\vartheta_s^{-1}\{(\ov{y},\ov{\xi})\in \Sigma_{r^s}:\ov{y}_0^{N-1}=\un{a}, \ov{\xi}\in (\alpha,\beta)\}
$
where $N\in\N$, $\un{a}$ is an admissible word of length $N$, and $\alpha,\beta\in\Q$.

Using that $r^s$ is independent of the past and that $\vartheta_s\circ\sigma_r^t=\sigma_{r^s}^t\circ\vartheta_s$,
one shows:
\begin{enumerate}[(a)]
\item $\sigma_r^{-t}[\mathfs W^{ss}_{\loc}]\subset\mathfs W^{ss}_{\loc}$, $\forall t>0$.
\item $\bigvee_{t>0}\sigma_r^t[\mathfs W^{ss}_{\loc}]=\mathfs B$ modulo $\mu$.
\end{enumerate}
Let $\mathfs W^{ss}:={\rm Tail}(\mathfs W^{ss}_{\loc})$.
By Theorem \ref{Thm-Rohlin-Sinai},
$\mathfs P(\sigma_r)\subset \mathfs W^{ss}$ modulo $\mu$.

Next we work  with an isomorphism $\vartheta_u:\Sigma_r\to\Sigma_{r^u}$ where $r^u$ is
independent of the future and $\vartheta_u\circ\sigma_r^t=\sigma_{r^u}^t\circ\vartheta_u$, $\forall t\in\R$.
Denoting points in $\Sigma_{r^u}$ also as $(\ov{x},\ov{\xi}):=\vartheta_u(x,\xi)$, we can define for each
$(x,\xi)\in\Sigma_{r}$ the set
$$
W^{su}_{\loc}(x,\xi):=\vartheta_u^{-1}\{(\ov{y},\ov{\xi}):\ov{y}_{-\infty}^0=\ov{x}_{-\infty}^0\}
$$
and $\mathfs W^{su}_{\loc}$ as the $\sigma$--algebra generated by the partition $\{W^{su}_{\loc}(x,\xi)\}$.
Similarly, $\sigma_r^{t}[\mathfs W^{su}_{\loc}]\subset\mathfs W^{su}_{\loc}$, $\forall t>0$, and
$\bigvee_{t>0}\sigma_r^{-t}[\mathfs W^{su}_\loc]=\mathfs B$ modulo $\mu$.
Let $\mathfs W^{su}:={\rm Tail}(\mathfs W^{su}_\loc)$.
Applying Theorem \ref{Thm-Rohlin-Sinai} to the inverse flow $\{\sigma_r^{-t}\}$ and using that it
has the same Pinsker $\sigma$--algebra as $\{\sigma_r^t\}$, we find that
$\mathfs P(\sigma_r)\subset \mathfs W^{su}$ modulo $\mu$.
We just proved:

\begin{thm}[\cite{Gurevich-K-flows,Ratner-Flows-K}]
Let $\sigma_r:\Sigma_r\to\Sigma_r$ be a TMF, and let $\mu$ be an ergodic $\sigma_r$--invariant
probability measure, not supported on a single orbit. Then
$
\mathfs P(\sigma_r)\subset \mathfs W^{ss}\cap \mathfs W^{su}$ modulo $\mu$.
\end{thm}

\begin{cor}\label{Corollary-Pinsker-Equiv-Rel}
Let $\sigma_r:\Sigma_r\to\Sigma_r$ be a TMF, and let $\mu$ be an ergodic $\sigma_r$--invariant
probability measure, not supported on a single orbit.
If $f:\Sigma_r\to\R$ is $\mathfs P(\sigma_r)$--measurable, then there is a set $X$ of full $\mu$--measure
s.t. for every $(x,\xi),(y,\eta)\in X$:
\begin{enumerate}[$(1)$]
\item If $(y,\eta)\in W^{ss}(x,\xi)$ then $f(x,\xi)=f(y,\eta)$.
\item If $(y,\eta)\in W^{su}(x,\xi)$ then $f(x,\xi)=f(y,\eta)$.
\end{enumerate}
\end{cor}

\begin{proof}
Recall the definitions of $W^{ss}(x,\xi)$ and $W^{su}(x,\xi)$ on page \pageref{Page-strong-manifold}.
We prove (1), and leave (2) to the reader. It is enough to prove this for $f=1_E$ where $E\in\mathfs P(\sigma_r)$.
Since $\mathfs P(\sigma_r)\subset\mathfs W^{ss}={\rm Tail}(\mathfs W^{ss}_\loc)$,
there is a sequence of sets $E_i\in \sigma_r^{-i}(\mathfs W^{ss}_{\loc})$ s.t.
$\mu(E\triangle E_i)=0$. The set
$
X:=\Sigma_r\setminus[(\cup_{i\geq 1}E\triangle E_i)\cup\{(x,\xi):x\text{ is
 pre-periodic}\}]
$
has full $\mu$--measure.

If $(x,\xi),(y,\eta)\in X$ with $(y,\eta)\in W^{ss}(x,\xi)$, then $\sigma_r^t(y,\eta)\in W^{ss}_{\loc}(\sigma_r^t(x,\xi))$
for $t$ large enough. In particular, this holds for some $t=i\in\N$.
We want to show that $(x,\xi)\in E\Leftrightarrow (y,\eta)\in E$. By symmetry, it is enough that
$(x,\xi)\in E\Rightarrow (y,\eta)\in E$.

Let $(x,\xi)\in E$. Then
$(x,\xi)\not\in E\triangle E_i\Rightarrow (x,\xi)\in E_i\Rightarrow\sigma_r^i(x,\xi)\in\sigma_r^i(E_i)\in \mathfs W^{ss}_{\loc}$.
The atom of $\mathfs W^{ss}_{\loc}$ which
contains $\sigma_r^i(x,\xi)$ is $W^{ss}_{\loc}((\sigma_r^i(x,\xi))$, so
$\sigma_r^i(y,\eta)\in W^{ss}_{\loc}((\sigma_r^i(x,\xi))\subset \sigma_r^i(E_i)\Rightarrow (y,\eta)\in E_i
\overset{!}{\Rightarrow}(y,\eta)\in E$ ($\overset{!}{\Rightarrow}$ is because $(y,\eta)\in X$).
\end{proof}

%
%
%
%

\subsection*{The Pinsker factor in the non-arithmetic case}
Let $\sigma:\Sigma\to\Sigma$ be a TMS. A H\"older continuous  $r:\Sigma\to\R$ is called {\em arithmetic}, if there are $\theta\in\R$, $\theta\neq 0$, and  $h:\Sigma\to S^1$  H\"older continuous s.t.  $e^{i\theta r}=h/h\circ\sigma$  \cite{Guivarch-Hardy}.

\begin{thm}\label{Thm-WM-K}
Let $\sigma_r:\Sigma_r\to\Sigma_r$ be a topologically transitive TMF, and let $\mu$ be an equilibrium
measure of a bounded H\"older continuous function with finite pressure. The following are equivalent:
\begin{enumerate}[$(1)$]
\item $r$ is not arithmetic.
\item $\mu$ is weak mixing.
\item $\mu$ is mixing.
\item $\mu$ has the K property, whence a trivial Pinsker factor.
\end{enumerate}
In particular, if one equilibrium measure of a bounded H\"older continuous function satisfies one of $(2)$--$(4)$,
then all equilibrium measures of bounded H\"older continuous functions satisfy all of $(2)$--$(4)$.
\end{thm}
If  $\Sigma$ is a subshift of finite type, then  the equivalences of (2)--(4)  are due to Ratner
\cite{Ratner-Flows-K} (a special case was done before by Gurevich \cite{Gurevich-K-flows}), and
$(1)\Leftrightarrow(2)$ is due to Parry \& Pollicott \cite[Prop. 6.2]{Parry-Pollicott-Asterisque}.

\begin{proof}
 $(4)\Rightarrow(3)$ by general theory, and $(3)\Rightarrow(2)$ is obvious. $(2)\Rightarrow(1)$ because if  $e^{i\theta r}=h/h\circ\sigma$ for some $\theta\neq 0$ and $h:\Sigma\to S^1$  continuous, then $F(x,\xi):=e^{-i\theta\xi}h(x)$ satisfies $F\circ \sigma_r^t=e^{-i\theta t}F$, $\forall t\in\R$. By the weak mixing assumption,  $F$ is constant $\mu$--a.e., whence everywhere (equilibrium measures of H\"older potentials on a topologically transitive TMF are globally supported).
Thus $\theta=0$.

\medskip
It remains to show that (1)$\Rightarrow$(4). We prove that if the Pinsker $\sigma$--algebra is not trivial
then $r$ is arithmetic. Assume that $\mathfs P(\sigma_r)$ is not trivial, and fix a bounded Pinsker-measurable
function $F$ that is not constant $\mu$--a.e. Let  $F_\delta:=\frac{1}{\delta}\int_0^\delta F\circ \sigma_r^t dt$.
Note that $F_\delta\xrightarrow[\delta\to 0^+]{L^1}F$, thus $F_\delta$ is not constant $\mu$--a.e
for any $\delta$ small enough. Fix one such $\delta$ and let $H:=F_\delta$. $H$ is a bounded
Pinsker-measurable function that is not constant $\mu$--a.e. for which the map
$t\mapsto (H\circ\sigma_r^t)(x,\xi)$ is continuous, $\forall(x,\xi)\in\Sigma_r$. We will use $H$ to prove that
$r$ is arithmetic.
Let $\nu$ be the induced measure of $\mu$.  Recall the definition of the cocycles $P^s, P^u$ (see Lemma \ref{Lemma-P}) and the measures $\nu^s_{x}$ on $W^s_\loc(x)$ defined in (\ref{nu-s}).

\medskip
\noindent
{\sc Claim  1.} {\em
There is a Borel set $E\subset\Sigma$ of full $\nu$--measure such that:\label{Page-Set-E}
\begin{enumerate}[$(1)$]
\item $E$ is $\sigma$--invariant and contains no pre-periodic points.
\item For every $(x,\xi),(y,\eta)$ s.t. $x,y\in E$:
\begin{enumerate}[$(2.1)$]
\item If $(y,\eta)\in W^{ss}(x,\xi)$ then $H(y,\eta)=H(x,\xi)$.
\item If $(y,\eta)\in W^{su}(x,\xi)$ then $H(y,\eta)=H(x,\xi)$.
\end{enumerate}
\item For every $x\in E$, $\nu^s_x(E^c)=\nu^u_x(E^c)=0$.
\end{enumerate}}

\medskip
\noindent
{\em Proof of Claim $1$.\/}
Let $E_0:=\{x\in\Sigma:x\textrm{ is not pre-periodic}\}$. $E_0$ has full $\nu$--measure, since $\nu$ is ergodic
and globally supported.
By Corollary \ref{Corollary-Pinsker-Equiv-Rel}, there is $X\subset\Sigma_r$ of full $\mu$--measure s.t.
(2) holds for all $(x,\xi),(y,\eta)\in X$.
Since  $\mu$ is equivalent to $\nu\times d\xi$,
$$
E_1:=\{x\in E_0: (x,\xi)\in X\textrm{ for Lebesgue a.e. }\xi\in[0,r(x))\}
$$
has full $\nu$--measure. We claim that $E_1$ satisfies (2).

We prove (2.1) and leave (2.2) to the reader.
Since $x,y\in E_1$, there is an open neighborhood $U\subset\R$ of $0$ s.t.
$(x,\xi+t),(y,\eta+t)\in X$ for Lebesgue a.e. $t\in U$. Find $t_k\xrightarrow[k\to\infty]{}0$ s.t.
$(x,\xi+t_k),(y,\eta+t_k)\in X$. By Lemma \ref{Lemma-P}(1), $(y,\eta+t_k)\in W^{ss}(x,\xi+t_k)$,
therefore by the definition of $X$ we have $H(x,\xi+t_k)=H(y,\eta+t_k)$.
Passing to the limit, and using that $t\mapsto (H\circ\sigma_r^t)(x,\xi)$ and $t\mapsto (H\circ\sigma_r^t)(y,\eta)$
are continuous, we conclude that $H(x,\xi)=H(y,\eta)$.

Now consider $E_2:=\bigcap_{n\in\Z}\sigma^n(E_1)$. $E_2$ has full $\nu$--measure and satisfies
(1)--(2) but not necessarily (3).
We define $E_3, E_4,\ldots$ by induction as
$$
E_n:=\{x\in E_{n-1}: \nu^s_{\s^k(x)}(E_{n-1}^c)=\nu^u_{\sigma^k(x)}(E_{n-1}^c)=0, \forall k\in\Z\}.
$$
$\{E_n\}$ is a decreasing sequence of $\sigma$--invariant sets of full $\nu$--measure each,
by Corollary  \ref{Cor-Prod-Structure}, thus
$
E:=\bigcap_{n=4}^\infty E_n
$
is $\sigma$--invariant set of full $\nu$--measure. $E$ satisfies (1)--(2) of the claim, since $E\subset E_0\cap E_1$.
To see that it also satisfies (3), just note that if $x\in E$ and $\tau=s,u$ then
$\nu^\tau_{x}(E^c)=\nu^\tau_{x}(\bigcup_{n\geq 3}E^c_n)=\lim\nu^\tau_{x}(E^c_n)=0$.

\medskip
\noindent
{\sc Construction of the holonomy group:}  Recall the {\em weak stable} and
{\em weak unstable manifolds} of $x\in\Sigma$:
\begin{enumerate}[$\circ$]
\item $W^{ws}(x):=\{y:\exists m,n\text{ s.t. }x_m^\infty=y_n^\infty\}$.
\item $W^{wu}(x):=\{y:\exists m,n\text{ s.t. }x_{-\infty}^m=y_{-\infty}^n\}$.
\end{enumerate}
The following constructions are motivated by \cite{Brin-Group-Extensions}:

\begin{enumerate}[$\circ$]
\item {\em $su$--path}: A finite sequence of points $\gamma=\<x^0,\ldots,x^n\>$ in $E$ s.t.
$x^i\in W^{w\tau_i}(x^{i-1})$ for some $\tau_i\in\{s,u\}$. If $x^0=x^n=x$, then $\gamma$ is called an
{\em $su$--loop at $x$}. \label{su-loop}
\item {\em Lift of $su$--path:}
Suppose $0\leq \theta<r(x^0)$. The {\em lift} of $\gamma=\<x^0,\ldots,x^n\>$ at $z_0:=(x^0,\theta)$ is
$\<z_0,\ldots,z_n\>\subset\Sigma_r$  where  $z_i=\sigma_r^{\theta+t_i}(x^i,0)$, and  $z_i\in W^{s\tau_i}(z_{i-1})$,
$i=1,\ldots,n$. The parameters $t_i$ are uniquely determined by the Bowen-Marcus condition, see Lemma \ref{Lemma-P}(1):
$t_0:=0$, $t_i=t_{i-1}+P^{\tau_i}(x^{i-1},x^i)$.
\item {\em Weight of $su$--loop:}
$
P(\gamma):=t_n=\sum_{i=1}^n P^{\tau_i}(x^{i-1},x^i).
$
\end{enumerate}

For $x\in E$, let $G_{x}':=\{P(\gamma):\textrm{$\gamma$ is an $su$--loop at $x$}\}$. We will show that there is a
closed subgroup $G\subset\R$ s.t. $G_x:=\ov{G_{x}'}=G$, $\forall x\in E$.

\medskip
\noindent
{\sc Holonomy group:} It is the closed subgroup $G\subset\R$ s.t.
$G_x=G$ for some (all) $x\in E$. \label{def-holonomy-group}

\medskip
We first show that $G_x=c\Z$ for some $c\neq 0$ independent of $x\in E$, and then use this to prove that $\exp[\frac{2\pi i}{c}r]$ is a multiplicative coboundary.

\medskip
\noindent
{\sc Claim 2:} {\em There exists $c\neq 0$ s.t.
 $G_x=c\Z$, $\forall x\in E$.}

\medskip
\noindent
{\em Proof of Claim $2$.\/} We divide the proof into few steps. Fix $x\in E$.

\medskip
\noindent
{\sc Step 1.\/} {\em  $G_x', G_x$ are additive subgroups of $\R$, and
$G_{\sigma(x)}'=G_{x}'$, $G_{\sigma(x)}=G_{x}$.}

\medskip
\noindent
{\em Proof.\/} It is enough to prove the claims for $G_x'$. $G_x'$ is an additive group:
\begin{enumerate}[$\circ$]
\item $G_{x}'+G_{x}'\subset G_{x}'$, because
$P(\gamma_1)+P(\gamma_2)=P(\gamma_1\vee\gamma_2)$ where $\gamma_1\vee\gamma_2$ is the concatenation of $\gamma_1$ and $\gamma_2$.
\item $G_{x}'\owns 0$, because $P(\<x,x\>)=0$.
\item $G_{x}'=-G_{x}'$, because $P(\<x^n,\ldots,x^0\>)=-P(\<x^0,\ldots,x^n\>)$.
\end{enumerate}

Now we show that $G_{\sigma(x)}'=G_{x}'$. Let $\gamma=\<x^0,\ldots,x^n\>$ be an $su$--loop at $x$,
and let $\sigma(\gamma):=\<\sigma(x^0),\ldots,\sigma(x^n)\>$. By Lemma \ref{Lemma-P}(2),
$P^{\tau_i}(\sigma(x^{i-1}),\sigma(x^{i}))-P^{\tau_i}(x^{i-1},x^i)=r(x^{i-1})-r(x^{i})$. Summing this over $i$
gives $P(\sigma(\gamma))-P(\gamma)=r(x^n)-r(x^0)=0$.

\medskip
\noindent
{\sc Step 2.\/} {\em  There is a closed subgroup $G\subset\R$ s.t. $G_x=G$, $\forall x\in E$.}

\medskip
\noindent
{\em Proof.\/} We claim that $x\mapsto G_x$ is constant on $E\cap [v]$, for every state $v$.
%
%
Take $x,y\in E\cap[v]$, and define $\pi_{xy}:W^s_{\loc}(x)\to W^s_{\loc}(y)$ by
$\pi_{xy}(\cdot)=[\cdot,y]$. $\pi_{xy}$ is measure-preserving:
$$
\nu^s_{x}\circ\pi_{xy}^{-1}=\nu\circ (p^s_{x})^{-1}\circ\pi_{xy}^{-1}=\nu\circ (\pi_{xy}\circ p^s_{x})^{-1}=
\nu\circ (p^s_{y})^{-1}=\nu^s_{y}.
$$
$E$ has full $\nu^s_{x}$--measure in $W^s_{\loc}(x)$. Since $\pi_{xy}$ is measure-preserving,
$\pi_{xy}[E\cap W^s_{\loc}(x)]$ has full $\nu^s_{y}$--measure in $W^s_{\loc}(y)$. Thus
$\pi_{xy}[E\cap W^s_{\loc}(x)]\cap E\neq\emptyset$, therefore $\exists z\in E\cap W^s_{\loc}(x)$ s.t.
$w:=[z,y]\in E\cap W^s_{\loc}(y)$. By the definition of the Smale product, $W^u_{\loc}(z)=W^u_{\loc}(w)$.
In summary, we found $z\in W^s_{\loc}(x)\cap E$, $w\in W^s_{\loc}(y)\cap E$ s.t. $W^u_{\loc}(z)=W^u_{\loc}(w)$.

Every element of $G_{x}'$ equals $P(\gamma)$ for some $su$--loop $\gamma$ at $x$.
Consider the concatenation
$
\gamma':=\<y,w,z,x\>\vee\gamma\vee\<x,z,w,y\>.
$
This is an $su$--loop at $y$ with $P(\gamma')=P(\<y,w,z,x,z,w,y\>)+P(\gamma)=P(\gamma)$.
Since $\gamma$ is arbitrary, this gives the inclusion
$G_x\subset G_y$.  By symmetry, $G_x=G_y$.

We see that for every $v$, there is a group $G_v$ s.t.  $G_x=G_v$, $\forall x\in E\cap [v]$.
Fix some state $v_0$. Since $\sigma:\Sigma\to\Sigma$ is topologically transitive, for any state $v$
there is an admissible path $v_0=a_0\to\cdots\to a_n=v$. The measure $\nu$ is globally supported,
thus we can take $z\in E\cap [\un{a}]$. By Step 1,
$G_{v_0}=G_{z}=G_{\sigma(z)}=\cdots=G_{\sigma^n(z)}=G_v$, whence $G_v=G_{v_0}$ for all vertices $v$.
This proves Step 2.

\medskip
\noindent
{\sc Step 3.\/} {\em $G$ equals $c\Z$ for some $c\in\R$.}

\medskip
\noindent
{\em Proof.\/}
$G$ is a closed additive subgroup of $\R$, so  either $G=\R$ or $G=c\Z$ for some $c\in\R$.
We will show that if $G=\R$ then $H$ is constant $\mu$--a.e., a contradiction.

We implement the classical Hopf argument. The key observation is that $H$ is  constant on the
intersection of the strong (un)stable manifolds of $\sigma_r$ with $E$, thanks to Claim 1(2).
Suppose  $\gamma=\<x^0,\ldots,x^n\>$ is an $su$--path, fix some $0\leq \theta<r(x^0)$, and let
$\<z_0,\ldots,z_n\>\subset\Sigma_r$ be the lift of $\gamma$ at $z_0:=(x^0,\theta)$.
Since $x^i\in E$, we have $H(z_0)=H(z_1)=\cdots=H(z_n)$.
In particular, if $x\in E$ and $\gamma$ is an $su$--loop at $x$, then $H(x,\theta)=(H\circ\sigma_r^{P(\gamma)})(x,\theta)$.

If $G=\R$ then the set of weights $P(\gamma)$ is dense in $\R$. Since $t\mapsto (H\circ\sigma_r^t)(x,\theta)$ is continuous, $H(x,\theta)=(H\circ\sigma_r^t)(x,\theta)$ for all $t\in\R$.
This proves that $H\circ\sigma_r^t=H$ on $\{(x,\theta)\in\Sigma_r:x\in E\}$.
Using that $\mu$ is ergodic (Corollary \ref{Cor-Eq-are-Ergodic}), we conclude that $H$ is constant $\mu$--a.e.,
a contradiction. Thus $G=c\Z$ for some $c\in\R$.

\medskip
\noindent
{\sc Step 4.\/}  $c\neq 0$.

\medskip
\noindent
{\em Proof.\/} Suppose  by contradiction that $G=\{0\}$.
We will show that  $r=U\circ\sigma-U$ for some $U:\Sigma\to\mathbb R$ continuous, and derive a contradiction.
Recall the definitions of $W^{ws}(x),W^{s}_{\loc}(x)$ on page \pageref{Page-Weak-(un)stable}.
Fix $x\in E$ and define $\widetilde U$ on $W^{ws}(x)\cap E$ by
$\widetilde U(y)=P^s(y,x)$. By Lemma \ref{Lemma-P}(3),
$$
\wt{U}(\sigma(y))-\wt{U}(y)=P^s(\sigma(y),x)+P^s(x,y)=P^s(\sigma(y),y)=r(y).
$$
Our plan is to show that $W^{ws}(x)\cap E$ is dense in $\Sigma$, and $\wt{U}$ is uniformly continuous on
$W^{ws}(x)\cap E$. Thus the unique continuous extension to $\Sigma$ satisfies   $U\circ\sigma-U=r$.

\medskip
\noindent
{\em Proof that $W^{ws}(x)\cap E$ is dense in $\Sigma$\/:} Let $C:={_{-n}[}v_{-n},\ldots,v_n]$ be a non-empty
cylinder in $\Sigma$.
Since $\sigma:\Sigma\to\Sigma$ is topologically transitive,  there is an admissible path
$v_n\to v_{n+1}\to\cdots\to v_{n+k}\to x_0$. Now proceed as follows:
\begin{enumerate}[$\circ$]
\item Pick some $w\in C$, and define $y$ by
$y_{-\infty}^{n}=w_{-\infty}^n$, $y_{n+1}^{n+k}=(v_{n+1},\ldots,v_{n+k})$, $y_{n+k+1}^\infty=x_0^\infty$.
Then $y\in W^{ws}(x)\cap C$, and there are integers $\ell,m>n$ s.t.
$\sigma^m(y)\in W^s_{\loc}(\sigma^\ell(x))\cap \sigma^m(C)$, whence $\sigma^m(C)\cap [x_\ell]\neq \emptyset$.
\item Necessarily
$
\nu^s_{\sigma^\ell(x)}(\sigma^m C)=\nu[(p^s_{\sigma^\ell(x)})^{-1}(\sigma^m C)]=\nu(\sigma^m(C)\cap [x_{\ell}])$.
Since $\nu$ is globally supported, $\nu^s_{\sigma^\ell(x)}(\sigma^m C)>0$.
\item Since $E$ is $\sigma$--invariant and $x\in E$, $\sigma^\ell(x)\in E$ and
$\nu^s_{\sigma^\ell(x)}(\sigma^m(C)\cap E)\neq 0$.
\item $\nu^s_{\sigma^\ell(x)}$ is supported on $W^s_{\loc}(\sigma^\ell(x))$, thus
$W^s_{\loc}(\sigma^\ell(x))\cap\sigma^m(C)\cap E\neq \emptyset$.
\item Therefore
$W^{ws}(x)\cap E\cap  C\supseteq \sigma^{-m}[W^s_{\loc}(\sigma^{\ell}(x))\cap\sigma^m(C)\cap E]\neq \emptyset.$
\end{enumerate}
We see that $W^{ws}(x)\cap E$ intersects every non-empty cylinder $C$ in $\Sigma$.

\medskip
\noindent
{\em Proof that $\wt{U}$ is uniformly continuous on $W^{ws}(x)\cap E$\/}:
Fix $y,z\in W^{ws}(x)\cap E$ s.t. $y\neq z$ and  $y_0=z_0$. We construct
 $y^1\in W^s_{\rm loc}(y)\cap E$ s.t.
\begin{enumerate}[(i)]
\item $z^1:=[y^1,z]\in W^{ws}(x)\cap E$,
\item $d(z,z^1)\leq d(y,z)$ and $d(z^1,y^1)\leq d(z,y)$,
\item $d(y,y^1)\leq 3d(y,z)$.
\end{enumerate}
\begin{figure}[hbt!]
\centering
\def\svgwidth{8cm}
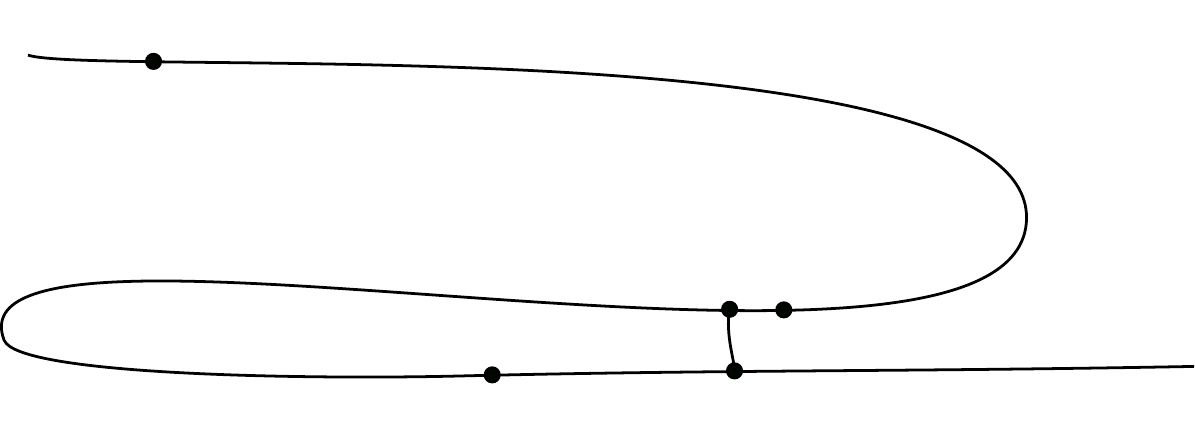
\label{figure coboundary}
\end{figure}
Here is how to do this. First, find $z^1\in W^s_{\loc}(z)\cap E$ arbitrarily close to $z$ s.t $y^1:=[z^1,y]\in W^s_{\loc}(y)\cap E$. Such points exist because $\nu^s_{z}(E^c)=0$, $\nu^s_{y}(E^c)=0$, $\nu^s_{z}$ has full support in $W^s_{\loc}(z)$,
and $\nu^s_{z}=\nu^s_{y}\circ\pi_{zy}$ for $\pi_{zy}(\cdot)=[\cdot,y]$.
Automatically $z^1=[y^1,z]$, and if $z^1$ is close enough to $z$, then $d(z^1,z)<d(z,y)$ and
$d(z^1,y)=d(z,y)$ (the first place where $z^1,y$ disagree is the first place where $z,y$ disagree).
Since $y^1=[z^1,y]$, $d(z^1,y^1)\leq d(z^1,y)=d(z,y)$, proving (ii).  Part (iii) follows from (ii) and the triangle inequality.

\medskip
Let $\gamma=\<y,z^1,y^1,y\>$. Using $y\in E$ and $G=\{0\}$, we have
\begin{align}\label{case 1 - sum of weights}
P^s(y,z^1)+P^u(z^1,y^1)+P^s(y^1,y)=0.
\end{align}
By Lemma \ref{Lemma-P}(3), $|\wt{U}(y)-\wt{U}(z^1)|=|P^s(y,z^1)|\leq |P^u(z^1,y^1)|+|P^s(y^1,y)|.$
Since $y^1\in W^u_{\loc}(z^1)$, $|P^u(z^1,y^1)|\leq C d(y,z)^\alpha$, where $C,\alpha$ are given by Lemma \ref{Lemma-P}(4). Similarly, $|P^s(y^1,y)|\leq 3^\alpha C d(y,z)^{\alpha}$. Thus
$|\wt{U}(y)-\wt{U}(z^1)|\leq 4C d(y,z)^{\alpha}$. Also by the cocycle equation,
$
|\wt{U}(z)-\wt{U}(z^1)|=|P^s(z,z^1)|\leq Cd(y,z)^\alpha
$.
It follows that $|\wt{U}(y)-\wt{U}(z)|<5C d(y,z)^{\alpha}$, proving that $\wt{U}$ is uniformly continuous on $W^{ws}(x)\cap E$.

\medskip
Therefore $\wt{U}$ extends continuously to a function $U:\Sigma\to\R$. Since $r=\wt{U}\circ\sigma-\wt{U}$ on
$W^{ws}(x)\cap E$, $r=U\circ\sigma-U$ on $\Sigma$. This cannot happen as it implies, by the Poincar\'e recurrence theorem,
that $\liminf r_n=\liminf [U\circ\sigma^n-U]<\infty$ a.e., whereas we know that $\inf r>0$, so $\liminf r_n=\infty$.
Thus $G\neq \{0\}$.

\medskip
\noindent
{\sc Claim 3:} {\em There exists $h:\Sigma\to S^1$ H\"older continuous  s.t. $\exp[\frac{2\pi i}{c}r]=h/h\circ\sigma$.}

\medskip
Let $\theta:=\frac{2\pi}{c}$, fix $x\in E$ and let $\wt{h}:W^{ws}(x)\cap E\to S^1$ by
 $\wt{h}(y):=\exp[-i\theta P^s(y,x)]$. By Lemma \ref{Lemma-P}(3), $\wt{h}/\wt{h}\circ\sigma=\exp[i\theta r]$ on
$W^{ws}(x)\cap E$.
The idea is to show that $\wt{h}$ is H\"older continuous on $W^{ws}(x)\cap E$ and then deduce as
in the previous proof that it extends H\"older continuously to a function $h:\Sigma\to S^1$.
The proof is the same as in the last step of Claim 2, except that one needs to replace (\ref{case 1 - sum of weights}) by
$$
\exp[i\theta (P^s(y,z^1)+P^u(z^1,y^1)+P^s(y^1,y))]=1.
$$
As before, this implies  that $\frac{\wt{h}(y)}{\wt{h}(z^1)}=e^{i\e_1}$ with $|\e_1|\leq 4C|\theta| d(y,z)^{\alpha}$,
and $\frac{\wt{h}(z^1)}{\wt{h}(z)}=e^{i\e_2}$ with $|\e_2|\leq C|\theta|d(y,z)^{\alpha}$. So $\frac{\wt{h}(y)}{\wt{h}(z)}=e^{i\e}$
with $|\e|\leq 5|\theta|d(y,z)^\alpha$, whence the H\"older continuity of $\wt{h}:W^{ws}(x)\to S^1$.

\medskip
Claim 3 completes the proof that if the Pinsker $\sigma$--algebra of $\sigma_r$ is not trivial then $r$ is arithmetic.
Equivalently, (1)$\Rightarrow$(4) in the statement of Theorem \ref{Thm-WM-K}, and this completes the proof of the theorem.
\end{proof}

\subsection*{The Pinsker factor in the arithmetic case} In the last section we saw that if the roof function is arithmetic,
then the Pinsker factor of every equilibrium measure of a bounded H\"older continuous potential is non-trivial.
In this section we show that in this case the Pinsker factor is isomorphic to  a rotational flow.
In fact we will show more, that the flow is isomorphic to the direct product of a Bernoulli flow and a rotational flow.

\begin{thm}\label{Thm-Non-non-arithmetic}
Let $\sigma_r:\Sigma_r\to\Sigma_r$ be a topologically transitive TMF s.t. $e^{i\theta r}=\frac{h}{h\circ\sigma}$ for some $\theta\neq 0$ and $h:\Sigma\to\R$ continuous.  There exists $p\in\N$ s.t. for every equilibrium measure $\mu$ of a bounded H\"older continuous potential with finite pressure, the following hold:
\begin{enumerate}[$(1)$]
\item $(\Sigma_r,\sigma_r,\mu)$ is isomorphic to a topologically transitive TMF with constant roof function
equal to $2\pi/\theta$.
\item $(\Sigma_r,\sigma_r,\mu)$ is isomorphic to the product of a Bernoulli flow  and a rotational flow with period ${2\pi p/\theta}$.
\item The Pinsker factor of $(\Sigma_r,\sigma_r,\mu)$ is isomorphic to a rotation with period ${2\pi p/\theta}$.
\end{enumerate}
\end{thm}

Before the proof of the theorem, let us prove that constant suspensions over Bernoulli automorphisms are the same as
the product of a Bernoulli flow and a rotational flow.

\begin{lem}\label{lemma-constant-versus-product}
Let $\mathsf{T}=(X,\mu,\{T^t\})$ be a measurable flow. The following are equivalent:
\begin{enumerate}[$(1)$]
\item $\mathsf{T}$ is isomorphic to a constant suspension over a Bernoulli automorphism.
\item $\mathsf{T}$ is isomorphic to the product of a Bernoulli flow and a rotational flow.
\end{enumerate}
\end{lem}

\begin{proof}
(1) $\Rightarrow$ (2). Assume that the roof function is $\equiv 1$. Then we can write $\mathsf{T}=(\Sigma_1,\mu,\{T^t\})$,
$T^t(x,s)=(S^{\lfloor t+s\rfloor}(x),t+s-\lfloor t+s\rfloor)$, where:
\begin{enumerate}[$\circ$]
\item $(\Sigma,\nu,S)$ is a Bernoulli automorphism.
\item $\Sigma_1$ is the suspension space over $\Sigma$ with roof function $\equiv 1$.
\item $\mu=\int_\Sigma\int_0^1 \delta_t dt d\nu(x)$.
\end{enumerate}
By Ornstein Theory, $(\Sigma,\nu,S)$ embeds into a Bernoulli flow $(\Sigma,\nu,\{S^t\})$,
see \cite{Ornstein-Imbedding-in-Flows}. Let $\{R^t\}$ be the rotational flow with period 1. We claim that $\mathsf{T}$
is isomorphic to $(\Sigma\times\mathbb T,\nu\times dt,\{S^t\times R^t\})$, the product of a Bernoulli flow and a rotational
flow. The conjugacy is the bijection $\rho:\Sigma_1\to\Sigma\times\mathbb T$, $\rho(x,s)=(S^s(x),s\text{ (mod 1)})$. First note that
$\rho$ is well-defined since $\rho(x,1)=(S^1(x),0)=(Sx,0)=\rho(Sx,0)$. Also:
\begin{align*}
&(\rho\circ T^t)(x,s)=\rho(S^{\lfloor t+s\rfloor}(x),t+s-\lfloor t+s\rfloor)=(S^{t+s}(x),t+s\text{ (mod 1)})\\
&=(S^t\times R^t)(S^s(x),s\text{ (mod 1)})=[(S^t\times R^t)\circ\rho](x,s).
\end{align*}
For all measurable $A\subset\Sigma$ and interval $I\subset\mathbb T$ not containing zero,
$(\mu\circ\rho^{-1})(A\times I)=\mu(A\times I)=\nu(A)\cdot|I|$, hence $\mu\circ\rho^{-1}=\nu\times dt$, which completes
the proof that $\rho$ is a conjugacy between $\mathsf{T}$ and $\{S^t\times R^t\}$.

\medskip
\noindent
(2) $\Rightarrow$ (1). With the same notation as above, assume that $\mathsf{T}=(\Sigma\times\mathbb T,\nu\times dt,\{S^t\times R^t\})$.
Then $\mathsf{T}$ is isomorphic to the suspension flow $(\Sigma_1,\mu,\{T^t\})$, where the basis dynamics is the Bernoulli
automorphism $(\Sigma,\nu,S^1)$. The conjugacy is the same $\rho$ as above, and the proof is analogous to (1) $\Rightarrow$ (2).
\end{proof}

\begin{proof}[Proof of Theorem \ref{Thm-Non-non-arithmetic}] Part (1) is the content of
\cite[Theorem 7.2]{Lima-Sarig}. Denote this TMF by
$\sigma_{\wt{r}}:\wt{\Sigma}_{\wt{r}}\to \wt{\Sigma}_{\wt{r}}$, with $\wt{r}\equiv 2\pi/\theta$.

Let $p$ denote the period of $\wt{\Sigma}$.
Recall from page \pageref{TMF-mixing-basis} that, using the spectral decomposition of
$\wt{\Sigma}$ \cite{Kitchens-Book}, $\sigma_{\wt{r}}:\wt{\Sigma}_{\wt{r}}\to \wt{\Sigma}_{\wt{r}}$
is topologically conjugate to a TMF $\sigma_{\wh{r}}:\wh{\Sigma}_{\wh{r}}\to \wh{\Sigma}_{\wh{r}}$ where
$\sigma:\wh{\Sigma}\to\wh{\Sigma}$ is topologically mixing, and $\wh{r}=\wt{r}_p=2\pi p/\theta=:\alpha$.

Let $\wh{\mu}$ be the measure on $\wh{\Sigma}_{\wh{r}}$ corresponding to $\mu$, and let
$\wh{\nu}$ be the induced measure of $\wh{\mu}$.
$\wh{\nu}$ is an equilibrium measure of a bounded H\"older continuous potential on $\wh{\Sigma}$
with finite pressure. Since $\sigma:\wh{\Sigma}\to \wh{\Sigma}$ is topologically mixing and $\wh{\Sigma}$ is
not a singleton, $\sigma:\wh{\Sigma}\to \wh{\Sigma}$ is Bernoulli \cite{Bowen-Bernoulli,Sarig-Bernoulli-JMD}.
By Lemma \ref{lemma-constant-versus-product}, $\sigma_{\wh{r}}:\wh{\Sigma}_{\wh{r}}\to \wh{\Sigma}_{\wh{r}}$
is isomorphic to the product of a Bernoulli flow and a rotational flow with period $\alpha$.

Since the Pinsker factor of a direct product is isomorphic to the direct product of the Pinsker
factors \cite[Prop. 4.4]{Thouvenot-Handbook}, and since Bernoulli flows have trivial Pinsker factor,
it follows that the Pinsker factor of $(\Sigma_r,\sigma_r,\mu)$ is isomorphic to $\mathfs P(R^t)=\mathfs P(R^1)=R^1$,
a rotation with period ${2\pi p}/\theta$.
\end{proof}

\section{The Bernoulli property}
We have proved so far that if $\sigma_r:\Sigma_r\to\Sigma_r$ is a topologically transitive TMF and $\mu$ is an equilibrium measure of a bounded H\"older continuous potential with finite pressure, then $(\Sigma_r,\sigma_r,\mu)$ is isomorphic to a Bernoulli flow times a rotational flow when $r$ is arithmetic, and $(\Sigma_r,\sigma_r,\mu)$ is a K flow when $r$ is
not arithmetic. The purpose of this section is to complete the picture and prove the following result.

\begin{thm}\label{Thm-Bernoulli}
Let $\sigma_r:\Sigma_r\to\Sigma_r$ be a topologically transitive TMF. If $r$ is not arithmetic,
then for every equilibrium measure $\mu$ of a bounded H\"older continuous function with finite pressure
$(\Sigma_r,\sigma_r,\mu)$ is a Bernoulli flow.
\end{thm}

\noindent
The theorem above strengthens Theorem \ref{Thm-WM-K} by saying that for equilibrium measures of bounded H\"older potentials with finite pressure, weak mixing is equivalent to  the Bernoulli property.

\subsection*{Review of general theory}

Let $(X,\mathfs B,\mu)$ be a non-atomic Lebesgue probability space, and let
$\alpha=\<A_1,\ldots,A_N\>$ and $\beta=\<B_1,\ldots,B_N\>$ be ordered partitions
of $(X,\mathfs B,\mu)$. Given $x\in X$, define $\alpha(x):=i$ if $x\in A_i$.

\medskip
\noindent
{\sc Partition distance:}
$d(\alpha,\beta):=\sum_{i=1}^N\mu(A_i\triangle B_i)=2\displaystyle\int 1_{[\alpha(x)\not=\beta(x)]}d\mu(x)$.

\medskip
Let $\{\alpha_i\}_1^n$ be a finite sequence of ordered partitions of $(X,\mathfs B,\mu)$,
and let $\{\beta_i\}_1^n$ be a finite sequence of ordered partitions of another
non-atomic Lebesgue probability space $(Y,\mathfs C,\nu)$.
Suppose that each partition has $N$ elements, say $\alpha_i=\<A_1^i,\ldots,A_N^i\>$
and $\beta_i=\<B_1^i,\ldots,B_N^i\>$.

\medskip
\noindent
{\sc Same distribution:} We say that $\{\alpha_i\}_1^n, \{\beta_i\}_1^n$ {\em have the same distribution},
and write $\{\alpha_i\}_1^n\sim \{\beta_i\}_1^n$, if
$$\mu[A_{i_1}^1\cap \cdots\cap A_{i_n}^{n}]=\nu[B_{i_1}^1\cap \cdots\cap B_{i_n}^{n}],\ \
\forall(i_1,\ldots,i_n)\in\{1,\ldots,N\}^n.$$

\medskip
This is equivalent to the existence of a measure preserving map $\theta:(X,\mathfs B,\mu)\to(Y,\mathfs C,\nu)$
such that $\theta[A_{i_1}^1\cap \cdots\cap A_{i_n}^{n}]=B_{i_1}^1\cap \cdots\cap B_{i_n}^{n}$
modulo $\nu$, $\forall(i_1,\ldots,i_n)\in\{1,\ldots,N\}^n$. This notion can be weakened in the following way.

\medskip
\noindent
{\sc $d$--bar distance:} The {\em $d$--bar distance} between $\{\alpha_i\}_1^n, \{\beta_i\}_1^n$
is
$$
\ov{d}(\{\alpha_i\}_{1}^n, \{\beta_i\}_{1}^n):=\inf\left\{\frac{1}{n}\sum_{i=1}^n d(\ov{\alpha}_i,\beta_i):\begin{array}{l}
\{\ov{\alpha}_i\}_1^n\text{ are ordered partitions of }\\
(Y,\mathfs C,\nu)\text{ s.t. }\{\ov{\alpha}_i\}_1^n\sim\{{\alpha}_i\}_1^n
\end{array} \right\}.
$$

\medskip
To understand how the $d$--bar distance weakens the notion of same distribution,
we first weaken the notion of measure preserving maps.

\medskip
\noindent
{\sc $\epsilon$--measure preserving map:} An invertible measurable map $\theta:(X,\mathfs B,\mu)\to (Y,\mathfs C,\nu)$
is called {\em $\epsilon$--measure preserving} if $\exists E\in\mathfs B$, $\mu(E)<\epsilon$,
s.t. $\left|\frac{\nu(\theta(A))}{\mu(A)}-1\right|\leq \epsilon$ for all $A\subset X\setminus E$ measurable.

%


\begin{lem}[\cite{Ornstein-Weiss-Geodesic-Flows}]\label{OW-Lemma}
If $\theta:(X,\mathfs B,\mu)\to (Y,\mathfs C,\nu)$ is $\epsilon$--measure preserving s.t.
$$
\frac{1}{n}\sum_{i=1}^n 1_{[\alpha_i(x)\neq \beta_i(\theta(x))]}\leq \epsilon
$$
on a set of measure $\geq 1-\epsilon$, then $\ov{d}(\{\alpha_i\}_1^n, \{\beta_i\}_1^n)\leq 16\epsilon$.
\end{lem}

\medskip
In other words, $\{\alpha_i\}_1^n, \{\beta_i\}_1^n$ are close in $d$--bar distance if there exists an $\epsilon$--measure
preserving map $\theta$ that matches $\alpha_i(x)$ and $\beta_i(\theta(x))$ on the average, for most points.
That is why the $d$--bar distance weakens the notion of same distribution.

We now explain the property we will use to prove an automorphism is Bernoulli.
Let $(X,\mathfs B,\mu,T)$ be an automorphism.
Given $A\in\mathfs B$ with $\mu(A)>0$, let $(A,\mathfs B_A, \mu_A)$ be the induced non-atomic Lebesgue
probability space, i.e. $\mathfs B_A:=\{B\cap A:B\in\mathfs B\}$ and $\mu_A(\cdot)=\mu(\cdot|A)$.
Every partition $\alpha$ of $(X,\mathfs B,\mu)$ defines a {\em conditional partition} $\alpha|A=\{C\cap A:C\in\alpha\}$
of $(A,\mathfs B_A, \mu_A)$.
Write {\em ``$\epsilon$--a.e. $A\in\alpha$"} when
refering to a property that holds for a collection of atoms of $\alpha$ whose union has measure $\geq 1-\epsilon$.

\medskip
\noindent
{\sc Very weak Bernoulli property\footnote{This is the formulation in \cite{Ornstein-Weiss-Geodesic-Flows}
and it implies the definition in \cite{Ornstein-Book}. The two definitions are equivalent for Bernoulli automorphisms,
since in this case every partition is VWB.}:}
$\alpha$ is called {\em very weak Bernoulli (VWB)}
if for every $\epsilon>0$ there is $N_0=N_0(\epsilon)$ s.t. for all $n\geq 0$ and $N'\geq N\geq N_0$ it holds
$$
\ov{d}\left(\{T^{-i}\alpha\}_1^n, \{T^{-i}\alpha|A\}_1^n \right)<\epsilon\
\textrm{ for $\epsilon$--a.e. }A\in\bigvee_{k=N}^{N'} T^{k}\alpha.
$$

\medskip
$\vee$ denotes the joining of partitions. Taking $A\in\bigvee_{k=N}^{N'} T^{k}\alpha$ means
that we are fixing the far past of $T$.


\begin{thm}[\cite{Ornstein-Imbedding-in-Flows,Ornstein-Isomorphism-Thm,Ornstein-Weiss-Geodesic-Flows}]
Let $\textsf{T}=(X,\mathfs B,\mu,\{T^t\})$ be a probability preserving measurable flow.
If for some $t$, $(X,\mathfs B,\mu,T^t)$ has an increasing sequence of VWB partitions which generate $\mathfs B$,
then $\textsf{T}$  is a Bernoulli flow.
\end{thm}

\subsection*{Construction of VWB partitions for  equilibrium measures
\cite{Ornstein-Weiss-Geodesic-Flows,Ratner-Flows-Bernoulli}}
Let $\sigma_r:\Sigma_r\to\Sigma_r$ be a  topologically transitive TMF. Throughout this section we assume that
$r$ is {\em not} arithmetic and independent of the {\em past} (which we can assume because of Lemma \ref{LemmaOneSided}).
Fix an equilibrium measure $\mu$ of a bounded H\"older continuous potential with finite pressure,
and let $\nu$ be the induced measure of $\mu$, i.e.
$\mu=\frac{1}{\int_\Sigma rd\nu}\int_\Sigma\int_0^{r(x)}\delta_{(x,t)}dtd\nu(x)$.

Let $\pi_1,\pi_2:\Sigma_r\to\Sigma$ be the projections on the first and second coordinates, respectively.
We now define three $\sigma$--algebras:
\begin{enumerate}[$\circ$]
\item $\alpha =$ partition of $\Sigma$ into cylinders of length one at the zeroth position.
$\bigvee_{i=0}^\infty \sigma^{-i}\alpha$ is the $\sigma$--algebra with information on the
coordinates $x_0^\infty$ of $x\in\Sigma$.
\item  $\mathfs F_{-n}:=\pi^{-1}_1(\bigvee_{i=-n}^\infty \sigma^{-i}\alpha)$, the $\sigma$--algebra
with information on $x_{-n}^\infty$ of $(x,t)\in\Sigma_r$.
\item $\mathfs H:=\pi^{-1}_2[\mathfs B(\R)]$, where $\mathfs B(\R)$ is the Borel $\sigma$--algebra of $\R$.
$\mathfs H$ is the $\sigma$--algebra with information on $t$ of $(x,t)\in\Sigma_r$.
\end{enumerate}

We will  abuse notation and write $\E_\mu(\cdot|x_{-n}^{\infty},t)$ instead of
$\E_\mu(\cdot|\mathfs F_{-n}\vee\mathfs H)(x,t)$ and  $\mu(E|x_{-n}^{\infty},t)$ instead of
$\E_\mu(1_E|\mathfs F_{-n}\vee\mathfs H)(x,t)$.
Since $r$ is independent of past coordinates, it can be easily checked that for all $n\geq 0$:
\begin{equation}\label{Conditional-Mu}
\mu(\cdot|x_{-n}^\infty,t)=1_{[r(x_0^\infty)>t]}(x,t)\cdot [\nu(\cdot|x_{-n}^\infty)\times\delta_t] \ \ \text{ for $\mu$--a.e. $(x,t)$.}
\end{equation}
Actually, there is a way to make sense of the right-hand-side for every $(x,t)$: use (\ref{g-function}) to define
$\nu(\cdot|x_0^\infty)$ for all $x$, and the identity $\nu\circ \sigma^{-n}=\nu$
to extend  to other $n$:
\begin{equation}\label{nu-identity}
\nu(E|x_{-n}^\infty):=\nu(\sigma^{-n}(E)|\sigma^{-n}(x)_0^\infty).
\end{equation}
Given an admissible word $\un{a}$, let $\rho(\un{a}):=\inf\{r(x):x_{-n}^{n}=\un{a}\}$. Let $0<\delta<1$, $n\geq 0$.
Consider the following definitions.

\medskip
\noindent
{\sc $(n,\delta)$--cube:} A set $C=\{(x,t): x_{-n}^{n}=\un{a}, t\in [\tau,\tau+\delta)\}$,
where  $\un{a}$ is an admissible word of length $2n+1$ and $\tau\geq 0$ s.t. $[\tau,\tau+\delta)\subset [0,\rho(\un{a}))$.

\medskip
\noindent
{\sc Canonical partition into $(n,\delta)$--cubes:} A finite or countable partition
whose atoms are $(n,\delta)$--cubes, with the exception of an atom of the form $\{(x,t):\rho(x_{-n}^n)\leq t<r(x)\}$
with measure $\leq\delta$.


\medskip
\noindent
{\sc Pseudo-canonical partition into $(n,\delta)$--cubes:}  A finite partition that can be refined to
a canonical partition into $(n,\delta)$--cubes.

\begin{lem}[\cite{Ratner-Flows-Bernoulli}]\label{Lemma-gamma-VWB}
If $n_0\geq 0$ and $0<t_0<\inf(r)$, then every pseudo-canonical partition into $(n_0,\delta_0)$--cubes
is very weak Bernoulli for $(\Sigma_r,\sigma_r^{t_0},\mu)$.
\end{lem}

\begin{proof}
This was proved (with different terminology and notation) in \cite{Ornstein-Weiss-Geodesic-Flows} for geodesic flows, and in \cite{Ratner-Flows-Bernoulli} for TMF built on subshifts of finite type.
What follows is a detailed exposition of the argument in \cite{Ratner-Flows-Bernoulli}, with some missing details added, and one (minor)  point clarified.

\medskip
Let $\gamma$ be a pseudo-canonical partition into $(n_0,\delta_0)$--cubes, and
take $N'\geq N>\frac{n_0}{t_0}\sup(r)$. Every $A\in\bigvee_{k=N}^{N'}\sigma_r^{t_0 k}\gamma$
is a countable union of sets of the form $$\{(x,t): x\in D_i, a_i(x)\leq t<b_i(x)\},$$ where
$D_i$ are cylinders in $\bigvee_{i=n_1(i)}^{n_2(i)} \sigma^i\alpha$ with
\begin{equation}\label{n-one-n-two}
\left\lfloor\tfrac{t_0 N}{\sup(r)}\right\rfloor-n_0-1\leq n_1(i)\leq n_2(i)\leq \left\lceil \tfrac{t_0 N'}{\inf(r)}\right\rceil+n_0+1,
\end{equation}
and  $a_i,b_i$ are independent of the past coordinates.

Fix $\epsilon>0$, and let $n\geq 0,\delta\in(0,1)$ to be determined later. Partition $\Sigma_r$ into finitely
many $(n,\delta)$--cubes $\mathcal C_{n,\delta}:=\{C_1,\ldots,C_m\}$ plus an additional ``error set"
with measure $\leq \delta$.

\medskip
\noindent
{\sc Step 1.\/} {\em $\exists N_0=N_0(n,\delta)>0$ s.t. for all $C\in\mathcal C_{n,\delta}$,
for all $N'\geq N\geq N_0$, and for $\delta$--a.e. $A\in\bigvee_{n=N}^{N'}\sigma_r^{t_0 n}\gamma$,
it holds $\left|\tfrac{\mu(A\cap C)}{\mu(A)\mu(C)}-1\right|<\delta$.}

\medskip
\noindent
{\em Proof.\/} Since $r$ is not arithmetic, $(\Sigma_r,\sigma_r^{t_0},\mu)$ is a K automorphism (Theorem \ref{Thm-WM-K}).
Now use Theorem \ref{Thm-K-Mixing} and the finiteness of $\mathcal C_{n,\delta}$.

\medskip
\noindent
{\sc Step 2.\/} {\em For all $A,C$ as in step 1, there is $(z,s)\in A\cap C$ s.t.
$\mu(A\cap C|z_{-n}^\infty, s)>0$ and $\mu(\cdot|z_{-n}^\infty,s)$ is non-atomic.
We choose one such pair for each $A,C$ and write $(z,s):=(z(A,C), s(A,C))$.}

\medskip
\noindent
{\em Proof.\/}
By Lemma \ref{Lemma-Non-Atomic} and (\ref{nu-identity}), $\nu(\cdot|z_{-n}^\infty)$ is non-atomic for $\nu$--a.e. $z$,
so $\mu(\cdot|z_{-n}^\infty,s)=1_{[r(x_0^\infty)>s]}[\nu(\cdot|z_{-n}^\infty)\times\delta_t]$ is non-atomic
for $\mu$--a.e. $(z,s)\in A\cap C$. Also $\bigl|\frac{\mu(A\cap C)}{\mu(A)\mu(C)}-1\bigr|<\delta<1\Rightarrow
\mu(A\cap C)>0\Rightarrow\mu(A\cap C|z_{-n}^\infty, s)>0$ for a subset $(z,s)\in A\cap C$ of positive $\mu$--measure.
Therefore there is $(z,s)\in A\cap C$ satisfying step 2.

\medskip
Given $\delta>0$, let us write $a=e^{\pm\delta}$ whenever $e^{-\delta}\leq a\leq e^\delta$.

\medskip
\noindent
{\sc Step 3.\/} {\em Given $\delta>0$, the following holds for all $n$ large enough.
If $(x,t), (z,s)\in C\in\mathcal C_{n,\delta}$,
then the map $\Theta_{x,t}^{z,s}:(C,\mu(\cdot|x_{-n}^\infty,t))\to (C,\mu(\cdot|z_{-n}^\infty,s))$, $\Theta_{x,t}^{z,s}(y,t)=(\vartheta({y}),s)$, where $\vartheta(y)=(y_{-\infty}^{-n-1},z_{-n}^\infty)$, has
Radon-Nikodym derivative equal to $e^{\pm\delta}$.}

\medskip
\noindent
{\em Proof.\/}
Write $C=B\times I$, where $B=_{-n}[b_{-n},\ldots,b_n]$ contains $x,z$
and $I$ is an interval of length $\delta$ containing $t,s$. The Radon-Nikodym derivative of $\Theta_{x,t}^{z,s}$
equals the Radon-Nikodym derivative of $\vartheta:(B,\nu(\cdot|x_{-n}^\infty))\to (B,\nu(\cdot|z_{-n}^\infty))$.
To estimate this latter derivative, let $B':=_{-(n+m)}[b_{-(n+m)},\ldots, b_n]\subset B$ be a cylinder,
and let $\epsilon_n:=\sum\limits_{k\geq n}\var_k(\log g)$.
By (\ref{g-function}) and (\ref{nu-identity}),
$$
\frac{\nu(B'|x_{-n}^\infty)}{\nu(\vartheta(B')|z_{-n}^\infty)}=\frac{g_{m}(b_{-(n+m)}^n,x_{n+1}^\infty)}{g_{m}(b_{-(n+m)}^n,z_{n+1}^\infty)}=e^{\pm\epsilon_n},
$$
thus $\nu(B'|x_{-n}^\infty)=e^{\pm\epsilon_n}\nu(\vartheta(B')|z_{-n}^\infty)$ for every cylinder  $B'\subset B$.
Since the cylinders generate the $\sigma$--algebra of Borel sets of $B$,
$\nu(E|x_{-n}^\infty)=e^{\pm\epsilon_n}\nu(\vartheta(E)|z_{-n}^\infty)$ for all Borel sets $E\subset B$,
hence the Radon-Nikodym derivative of $\vartheta$ equals $e^{\pm\epsilon_n}$.
Since $\log g$ is H\"older continuous, $\epsilon_n\xrightarrow[n\to\infty]{}0$, thus $\epsilon_n<\delta$ for all $n$
large enough.

\medskip
\noindent
{\sc Step 4.\/} {\em For all $A,C,(z,s)$ as in steps {\rm 1--2},
there is an invertible bi-measurable map
$
\Psi:(C,\mu(\cdot|z_{-n}^\infty,s))\to (A\cap C, \mu(\cdot|z_{-n}^\infty,s))
$ with constant Radon-Nikodym derivative. Call the constant $D(A,C)$.
}

\medskip
\noindent
{\em Proof.\/}
Any two non-atomic Lebesgue probability spaces are measure theoretically isomorphic.
$(C,\mu(\cdot|z_{-n}^\infty,s))$ and $(A\cap C,\mu(\cdot|z_{-n}^\infty,s))$ are non-atomic Lebesgue
measure spaces, so instead of an isomorphism there is an invertible bi-measurable map
$\Psi :(C,\mu(\cdot|z_{-n}^\infty,s))\to (A\cap C, \mu(\cdot|z_{-n}^\infty,s))$
with constant Radon-Nikodym derivative equal to $D(A,C):=\frac{\mu(A\cap C|z_{-n}^\infty,s)}{\mu(C|z_{-n}^\infty,s)}$.

\medskip
Let $\Omega:=\bigcup_{i=1}^m C_i$, $\mu(\Omega)>1-\delta$.

\medskip
\noindent
{\sc Step 5.\/} {\em If $\delta$ is sufficiently small, $n$ is sufficiently large, and $N_0=N_0(n,\delta)$ as in step 1,
then for all $N'\geq N\geq N_0$,  for $\delta$--a.e. $A\in\bigvee_{k=N}^{N'}\sigma^{t_0k}\gamma$,
there is a map $\Xi:(\Sigma_r,\mu)\to (A,\mu(\cdot|A))$ s.t.
\begin{enumerate}[$(1)$]
\item $\Xi(x,t)=(y,t)$ with $y_{-n}^\infty=x_{-n}^\infty$ for $(x,t)\in\Omega$,
\item $\Xi$ is invertible and bi-measurable,
\item $\Xi$ is $5\delta$--measure preserving.
\end{enumerate}
}

\medskip
\noindent
{\em Proof.\/} For each $A,C, (z,s)$ as in steps 1--2, define $\Xi\restriction_C:C\to A\cap C$ by
$$
\Xi(x,t)=(\Theta_{z,s}^{x,t}\circ\Psi\circ\Theta_{x,t}^{z,s})(x,t).
$$
Now define $\Xi$ on $\Sigma_r\setminus\Omega$ to take values on
$A\setminus \Omega$ via a bijective measure preserving map. Thus (1) holds\footnote{Our construction
of $\Xi$ differs from \cite{Ratner-Flows-Bernoulli}, since it is not clear to us that her construction
leads to a measurable map. Instead, we follow the construction used in \cite{Ornstein-Weiss-Geodesic-Flows}.}.

To prove (2), first note that $C\in \mathfs F_{-n}\vee\mathfs H$, hence we can write
$$
\mu\upharpoonright_C=\const\int_C \mu(\cdot|x_{-n}^\infty,t)d\mu(x,t).
$$
By steps 3--4, $\Xi\restriction_C:(C,\mu(\cdot|x_{-n}^\infty,t))\to (A\cap C, \mu(\cdot|x_{-n}^\infty,t))$ is an
absolutely continuous bijection, thus $\Xi\restriction_C:(C,\mu)\to (A\cap C,\mu)$ is bijective a.e., which gives (2).

Let us now prove (3). By steps 3--4, $\Xi\restriction_C:(C,\mu(\cdot|x_{-n}^\infty,t))\to (A\cap C, \mu(\cdot|x_{-n}^\infty,t))$
has Radon-Nikodym derivative $e^{\pm2\delta}D(A,C)$, thus if $E\subset C$ is measurable then
\begin{align*}
\mu(\Xi(E))&=\const\int_C \mu(\Xi(E)|x_{-n}^\infty,t)d\mu(x,t)\\
&=\const e^{\pm 2\delta}D(A,C)\int_C\mu(E|x_{-n}^\infty,t)d\mu(x,t)=\const e^{\pm 2\delta}D(A,C)\mu(E).
\end{align*}

Therefore $\Xi:C\to A\cap C$ is absolutely continuous with Radon-Nikodym derivative equal to $e^{\pm2\delta}K$ for some
constant $K=K(A,C)$. Since $\Xi$ is a bijection a.e., $K=e^{\pm 2\delta}\tfrac{\mu(A\cap C)}{\mu(C)}$.
If $\delta$ is so small that $1-\delta>e^{-2\delta}$, step 1 gives that
$
K=e^{\pm 4\delta}\mu(A).
$
Since $C\in\mathcal C_{n,\delta}$ is arbitrary, $\Xi\restriction_\Omega:(\Omega,\mu)\to (A\cap \Omega,\mu)$
has Radon-Nikodym derivative equal to $e^{\pm 4\delta}\mu(A)$.
After normalizing the measure of $A$, we find that the Radon-Nikodym derivative of
$
\Xi:(\Sigma_r,\mu)\to \bigl(A,\mu(\cdot|A)\bigr)
$ equals $e^{\pm4\delta}$ on $\Omega$.
If $\delta$ is so small $e^{4\delta}<1+5\delta$, we conclude that $\Xi$ is $5\delta$--measure preserving.

\medskip
\noindent
{\sc Step 6.\/} {\em If $\delta$ is sufficiently small and $n$ is sufficiently large, then for all
$m\geq 0$, for all $N'\geq N\geq N_0(n,\delta)$, and for $\delta$--a.e. $A\in\bigvee_{i=N}^{N'}\sigma_r^{it_0}\gamma$,
$$
\frac{1}{m}\#\{1\leq i\leq m:\sigma_{r}^{it_0}(x,t),\sigma_r^{it_0}(\Xi(x,t))\text{ are in different $\gamma$--atoms}\}< \epsilon
$$
holds for a set $(x,t)\in\Sigma_r$ of measure $\geq 1-\delta$.
}

\medskip
\noindent
{\em Proof.\/} This follows, as in \cite{Ornstein-Weiss-Geodesic-Flows,Ratner-Flows-Bernoulli},
from the fact that $\Xi(x,t)=(y,t)$ with $y_{-n}^\infty=x_{-n}^\infty$ for $(x,t)\in\Omega$.
Let us recall the argument.

Let $\wh{\gamma}$ denote the (countable) canonical partition into $(n_0,\delta_0)$--cubes which refines $\gamma$,
and assume that $n>n_0$.
If $\sigma_r^{it_0}(x,t), \sigma_r^{i t_0}(y,s)$ belong to different $\gamma$--atoms, then they belong to
different $\wh{\gamma}$--atoms. At least one of these atoms is an $(n_0,\delta_0)$--cube of the form $C:=B\times [a,a+\delta_0)$ with $B\in\bigvee_{j=-n_0}^{n_0}\sigma^j \alpha$. Using that $n>n_0$, that $r$ is independent of the past, and
that $x_{-n}^\infty=y_{-n}^\infty$, we get that $\sigma_r^{it_0}(x,t)$ belongs to
$\sigma_r^{[-\delta,\delta]}(C):=\bigcup_{|\theta|<\delta}\sigma_r^{\theta}(C)$. Let
$\sigma_r^{[-\delta,\delta]}(\wh{\gamma})$ be the union of all $\sigma_r^{[-\delta,\delta]}(C)$, $C\in\wh{\gamma}$ a $(n,\delta)$--cube.

Defining
$Z_m(x,t):=\frac{1}{m}\#\{1\leq i\leq m: \sigma_{r}^{it_0}(x,t),\sigma_r^{it_0}(\Xi(x,t))$ are in different
$\gamma$--atoms$\}$ and
$Y_m(x,t):=\sum_{i=1}^m 1_{\sigma_r^{[-\delta,\delta]}(\wh{\gamma})}(\sigma_r^{i t_0}(x,t))$,
the previous paragraph and the Markov inequality imply that
$$\mu[Z_m\geq \epsilon]\leq \mu[Y_m\geq \epsilon m]\leq\frac{1}{\epsilon m}\int Y_m d\mu\leq
\epsilon^{-1}\mu[\sigma_r^{[-\delta,\delta]}(\wh{\gamma})].$$
If we choose $\delta$ so small that $\mu[\sigma_r^{[-\delta,\delta]}(\wh{\gamma})]<\epsilon^2$, then
$\mu[Z_m\geq \epsilon]<\epsilon$ as required.

\medskip
\noindent
{\it Completion of the proof of Lemma \ref{Lemma-gamma-VWB}.} Given $\epsilon>0$,
let $\delta$ be sufficiently small and $n$ sufficiently large s.t. steps 1--6 hold, and $100\delta<\epsilon$.
By Lemma \ref{OW-Lemma}, for all $m\geq 0$, for all $N'\geq N\geq N_0(n,\delta)$, and for
$\delta$--a.e. $A\in\bigvee_{k=N}^{N'}\sigma^{kt_0}\gamma$ it holds
$\ov{d}(\{\sigma_r^{-i t_0}\gamma\}_1^m,\{\sigma_r^{-i t_0}\gamma|A\}_1^m)<16\times 5\delta<\epsilon.
$
Since $\epsilon>0$ is arbitrary, $\gamma$ is VWB.
\end{proof}

\subsection*{Proof of Theorem \ref{Thm-Bernoulli}} Fix $t_0\neq 0$, and construct an increasing sequence of pseudo-canonical partitions into $(n_k,\delta_k)$--cubes, with $n_k\to\infty$ and $\delta_k\to 0$.
This sequence of partitions generates the full $\sigma$--algebra of $\Sigma_r$. Since each of these
pseudo-canonical partitions is VWB for $\sigma_r^{t_0}$ (Lemma  \ref{Lemma-gamma-VWB}), it follows
from Ornstein Theory \cite{Ornstein-Imbedding-in-Flows,Ornstein-Isomorphism-Thm,Ornstein-Weiss-Geodesic-Flows}
that $(\Sigma_r,\sigma_r,\mu)$ is a Bernoulli flow.\hfill $\Box$

\vspace{.5cm}
\noindent
{\bf\large Part 2. Smooth Flows in Three Dimensions}

\section{Proof of Theorem \ref{Thm-Main}}\label{Section-Thm-Main}
Let $M$ be a three dimensional compact $C^\infty$ Riemannian manifold, let $X:M\to TM$ be a non-vanishing $C^{1+\epsilon}$ vector field,
and let $\{T^t\}$ be the flow on $M$ generated by $X$. Let $F:M\to\R$ be a bounded H\"older continuous function, and let $\nu$
be an equilibrium measure of $F$.
Our task is to show that
 $\nu$ has at most countably many ergodic components $\nu_i$ with positive entropy, and that
 $\{T^t\}$ is Bernoulli up to a possible period with respect to each $\nu_i$.

That $\nu$ has at most countably many ergodic components with positive entropy was proved
in \cite{Lima-Sarig} in the special case $F\equiv 0$. The same proof works for general bounded
H\"older continuous $F$ almost verbatim. Let us recap the idea.
For a fixed $\chi>0$, we prove that $F$ has at most countably many
$\chi$--hyperbolic\footnote{$\nu$ is $\chi$--hyperbolic if $\nu$--a.e. point has one
Lyapunov exponent $>\chi$ and another $<-\chi$.} ergodic
equilibrium measures. This happens because every ergodic equilibrium measure
on a TMS is carried by a topologically transitive TMS.
If there were uncountably many $\chi$--hyperbolic equilibrium measures for $F$,
then their convex combination would generate a $\chi$--hyperbolic equilibrium measure on
a TMS with uncountably many ergodic components. Taking the union over $\chi_n=1/n$ gives countability.

It remains to show that if $\nu$ is ergodic with positive entropy, then $\nu$ is Bernoulli up to a possible period.
Given a TMF $\sigma_r:\Sigma_r\to\Sigma_r$, let
$$
\Sigma_r^\#:=\{(x,t)\in\Sigma_r:\text{$\{x_i\}_{i>0}$, $\{x_i\}_{i<0}$ have constant subsequences}\}.
$$
By the Poincar\'e recurrence theorem, $\Sigma_r^\#$ has full measure for every $\sigma_r$--invariant probability measure.

Apply \cite[Theorem 1.2]{Lima-Sarig} to the flow $(M,\nu,\{T^t\})$ to get a TMF $\sigma_r:\Sigma_r\to \Sigma_r$ and a H\"older continuous map $\pi_r:\Sigma_r\to M$ s.t.:
\begin{enumerate}[(1)]
\item $\pi_r\circ \sigma_r^t=T^t\circ\pi_r$, $\forall t\in\mathbb R$.
\item $\pi_r[\Sigma_r^\#]$ has full $\nu$--measure.
\item $\pi_r:\Sigma_r^\#\to M$ is finite-to-one.
\end{enumerate}

Notice that $\Phi:=F\circ\pi_r$ is a bounded H\"older continuous function. Arguing as in \cite[Theorem 6.2]{Lima-Sarig}, one can prove
that $\Phi$ has an ergodic equilibrium measure $\mu$ s.t. $\mu\circ\pi_r^{-1}=\nu$. By ergodicity, $\mu$ is carried by a topologically transitive
TMF of $\Sigma_r$. By Theorems \ref{Thm-K-Mixing} and \ref{Thm-Bernoulli}, $\mu$ is Bernoulli up to a period.
Therefore $(M,\nu,\{T^t\})$ is a finite-to-one factor of a flow which is Bernoulli up to a period, so it is enough to prove the lemma below.

\begin{lem}
If a measurable flow is Bernoulli up to a period, then so are its finite-to-one factors.
\end{lem}

\begin{proof}
Suppose $\pi:X\to Y$ is a finite-to-one factor map between $\textsf{T}=(X,\mu,\{T^t\})$ and  $\textsf{S}=(Y,\eta,\{S^t\})$. Suppose $\textsf{T}$ is Bernoulli up to a period.

If $\textsf{T}$ is Bernoulli, then $T^1$ is Bernoulli. Since factors of Bernoulli automorphisms are Bernoulli automorphisms
\cite{Ornstein-Factor-Bernoulli}, $S^1$ is a Bernoulli automorphism. By  \cite{Ornstein-Imbedding-in-Flows}, $\textsf{S}$ is a Bernoulli flow.

Assume that $\textsf{T}$ is isomorphic to a Bernoulli flow times a rotational flow. By Lemma \ref{lemma-constant-versus-product}, it is enough to prove the claim below.

\medskip
\noindent
{\sc Claim:} If $\mathsf{T}$ is a constant suspension over a Bernoulli automorphism,
then $\mathsf{S}$ is a constant suspension over a Bernoulli automorphism.

\medskip
\noindent
{\em Proof.\/}
Assume without loss of generality that the roof function of $\mathsf{T}$ is $\equiv 1$, i.e. $\mathsf{T}=(\Sigma_1,\mu,\{T^t\})$
where $T^t(x,s)=(\tau^{\lfloor t+s\rfloor}(x),t+s-\lfloor t+s\rfloor)$, $\mu=\int_\Sigma\int_0^1 \delta_{(x,t)}dtd\mu_0(x)$,
and $(\Sigma,\mu_0,\tau)$ is a Bernoulli automorphism.

\medskip
Let $Y_0:=\pi(\Sigma\times\{0\})$. We claim that $Y_0$ is a Poincar\'e section for $\textsf{S}$. For each $y\in\pi(\Sigma_1)$,
let $I_y:=\{t>0: S^t(y)\in Y_0\}$.
\begin{enumerate}[$\circ$]
\item $I_y\neq\emptyset$: $y=\pi(x,s)\Rightarrow 1-s\in I_y$.
\item $I_y\cap (0,1)$ is finite: if $S^{t_n}(y)=\pi(x_n,0)$ for infinitely many $t_n,x_n$, then $y$ has infinitely many pre-images
$(\tau^{-1}(x_n),1-t_n)$.
\item $I_y$ is infinite: $y=\pi(x,t)\Rightarrow S^{n-t}(y)=\pi(\tau^n(x),0)\Rightarrow n-t\in I_y$, $\forall n>0$.
\end{enumerate}
By symmetry, $\{t<0:S^t(y)\in Y_0\}$ is non-empty, infinite, and has no accumulation points. Therefore $Y_0$ is a Poincar\'e section for $\textsf{S}$.

\medskip
$r(y):=\min\{t>0: S^t(y)\in Y_0\}$ is well-defined and positive $\eta$--a.e. Using that $\pi$ commutes $\mathsf{T}$ and $\mathsf{S}$,
we have $r\circ S^1=r$, thus $r$ is constant $\eta$--a.e.
Let $U:Y_0\to Y_0$, $U(y)=S^{r(y)}(y)$, and let $\eta_0:=(\mu_0\times\delta_0)\circ \pi^{-1}$. $\textsf{S}$ is a constant
suspension over $(Y_0,\eta_0,U)$. But $(Y_0,\eta_0,U)$ is a factor of $(\Sigma,\mu_0,\tau)$, hence it is a Bernoulli automorphism.


%
\end{proof}

\section{Reeb flows}\label{Section-Reeb}

Let $M$ be a compact three dimensional smooth Riemannian manifold without boundary, equipped with the following objects \cite{Geiges-Contact-Topology}:

\medskip
\noindent
{\sc A Contact form:} A smooth $1$--form $\alpha$ on $M$ s.t. $\omega:=\alpha\wedge d\alpha$ is a volume form. In this case,
$\ker(d\alpha)_x:=\{v\in T_x M:d\alpha(v,\cdot)\equiv 0\}$ is one-dimensional for all $x\in M$.

\medskip
\noindent
{\sc The Reeb vector field} (of $\alpha$): The unique vector field $X$ s.t. $X_x\in\ker(d\alpha)_x$ and $\alpha(X_x)=1$
for all $x\in M$. Necessarily $i_X\omega=d\alpha$.

\medskip
\noindent
{\sc The Reeb flow} (of  $\alpha$):
The flow $\{T^t\}$ generated by the Reeb vector field of $\alpha$. This is a smooth flow with positive speed.  $\{T^t\}$  preserves $\alpha$, i.e.
$\alpha(dT^tv)=\alpha(v)$ for all  $v$, since $\frac{d}{dt}(T^t)^\ast\alpha=(T^t)^\ast L_X\alpha=(T^t)^\ast[di_X\alpha+i_X(d\alpha)]=(T^t)^\ast[0+0]=0$.

\noindent
This setup covers geodesic flows of surfaces, and Hamiltonian flows of a system with two degrees of freedom restricted to  regular  energy surfaces  \cite{Abraham-Marsden-Mechanics}.

\medskip
We now add the assumption that $\{T^t\}$ has positive topological entropy.
Let  $\mu$ be an ergodic equilibrium measure of a H\"older continuous potential with positive metric entropy.
By Theorem \ref{Thm-Main}, $\textsf T=(M,\mu,\{T^t\})$ is Bernoulli up to a period. We will show that $\textsf T$ is Bernoulli.  A similar result for absolutely continuous measures is due to Katok \cite[Theorem 3.6]{Katok-Burns-Infinitesimal}.

\medskip
In dimension three, every ergodic invariant probability measure with positive metric
entropy is non-uniformly hyperbolic \cite{Ruelle-Entropy-Inequality},
hence there is a $\textsf{T}$--invariant set $M_0\subset M$ of full $\mu$--measure
s.t. for all $x\in M_0$ we have $T_xM=E^u(x)\oplus E^s(x)\oplus X(x)$
where $E^u(x),E^s(x)$ are one-dimensional linear subspaces satisfying:
\begin{enumerate}[$\circ$]
\item $\lim\limits_{t\to\pm\infty}\frac{1}{t}\log\|dT^t_x v\|<0$ for all non-zero $v\in E^s(x)$,
\item $\lim\limits_{t\to\pm\infty}\frac{1}{t}\log\|dT^{-t}_x v\|<0$ for all non-zero $v\in E^u(x)$,
\item $dT^t_x E^s(x)=E^s(T^t(x))$ and $dT^t_x E^u(x)=E^u(T^t(x))$, $\forall t\in\R$,
\item There is an immersed smooth curve $W^s(x)\ni x$ s.t. $T_y W^s(x)=E^s(y)$ and
$d(T^t(x),T^t(y))\xrightarrow[t\to\infty]{}0$, $\forall y\in W^s(x)$.
An analogous result holds for $W^u(x)$.
\end{enumerate}
See \cite[\S8.2]{Barreira-Pesin-Non-Uniform-Hyperbolicity-Book}.

\medskip
\noindent
{\sc Quadrilateral:} A {\em quadrilateral} is a closed  embedded curve $\gamma:[0,1]\to M$
s.t. there are four distinct points $x_0,x_1,x_2,x_3\in M_0$ with:
\begin{enumerate}[$\circ$]
\item $x_{i+1}\in W^{\tau_i}(x_i)$ for some $\tau_i\in\{s,u\}$ (here $x_4=x_0$),
\item If $\gamma(t_i)=x_i$, then $\gamma\restriction_{(t_i,t_{i+1})}$ is smooth
with $\gamma'(t)\in E^{\tau_i}(\gamma(t))$, $\forall t\in(t_i,t_{i+1})$.
\end{enumerate}

Quadrilaterals are the four-legged geometrical version of $su$--loops considered in page \pageref{su-loop}.
Call $x_0,\ldots,x_3$ the {\em vertices} of the quadrilateral. The next lemma is standard.

\begin{lem}\label{lemma-quadrilaterals}
Let $\textsf{T}=(M,\mu,\{T^t\})$ be as above. Then $E^s(x)\oplus E^u(x)={\rm ker}(\alpha_x)$, $\forall x\in M_0$.
In particular, if $\gamma$ is a quadrilateral then $\int_\gamma \alpha=0$.
\end{lem}

\begin{proof}
Let $v\in E^s(x)$. By the $\textsf{T}$--invariance of $\alpha$,
$\alpha(v)=\lim_{t\to+\infty}\alpha(dT^t v)=0$, hence $E^s(x)\subset{\rm ker}(\alpha_x)$. Since contact forms are non-degenerate, $\dim{\rm ker}(\alpha_x)=2$ whence  $E^s(x)\oplus E^u(x)={\rm ker}(\alpha_x)$.
If $\gamma$ is a quadrilateral then $\gamma'(t)\in E^s(\gamma(t))\oplus E^u(\gamma(t))$ except at the vertices,
therefore $\int_\gamma\alpha=\int_0^1 \alpha(\gamma'(t))dt=0$.
\end{proof}

\begin{proof}[Proof of Theorem \ref{Thm-Reeb}.]
Using the same notation of section \ref{Section-Thm-Main}, there is a topological Markov flow $\sigma_r:\Sigma_r\to \Sigma_r$
and a H\"older continuous map $\pi_r:\Sigma_r\to M$ s.t.:
\begin{enumerate}[(1)]
\item $\pi_r\circ \sigma_r^t=T^t\circ\pi_r$, $\forall t\in\mathbb R$,
\item $\pi_r[\Sigma_r^\#]$ has full $\mu$--measure,
\item $\pi_r:\Sigma_r^\#\to M$ is finite-to-one,
\item $(\Sigma_r,\mu\circ\pi_r^{-1},\sigma_r)$ is Bernoulli up to a period, and it has a period iff
$r$ is arithmetic, iff the holonomy group equals $c\mathbb Z$ for some $c>0$ (see Theorem \ref{Thm-WM-K}).
\end{enumerate}
Assume by way of contradiction that there is a period.

Let $\nu$ be the induced measure of
$\mu\circ\pi_r^{-1}$, then $\nu$ is globally supported on $\Sigma$ and has local product structure
(Theorem \ref{Prop-Prod-Strctr-measure}).
Let $\nu^s_x,\nu^u_x$ be the projection measures of $\nu$, as in (\ref{nu-s}). These are globally supported measures
on $W^s_{\rm loc}(x),W^u_{\rm loc}(x)$.

Let $E$ be the set constructed on page \pageref{Page-Set-E}, then the holonomy group  equals the closure of the set of  weights of $su$--loops with vertices in $E$. The assumption that $\textsf T$ has a period translates to the holonomy group being equal to $c\Z$ with $c>0$.

The Bowen-Marcus cocycle $P^\tau(\cdot,\cdot)$ is H\"older continuous (Lemma \ref{Lemma-P}(4)),
therefore $\exists\delta>0$ s.t. $d(x,y)<\delta\Rightarrow P^\tau(x,y)<c/5$
wherever defined.  We claim there exist
four distinct points $w_0,\ldots,w_3\in E$ s.t. $d(w_i,w_j)<\delta$ for all $i,j$ and
$\gamma_0=\langle w_0,w_1,w_2,w_3,w_0\rangle$ is a $su$--loop with $P(\gamma_0)=0$.
This can be done as follows:
\begin{enumerate}[$\circ$]
\item Fix $x,y\in E$ s.t. $d(x,y)<\delta$ and $y\notin W^s_{\rm loc}(x)$.
\item By Claim 1 of Theorem \ref{Thm-WM-K}, $\nu^s_x(E^c)=\nu^s_y(E^c)=0$, hence
$\{w\in E\cap W^s_{\rm loc}(x):[w,y]\in E\}$ has full $\nu^s_x$--measure.
\item $\nu^s_x$ is globally supported on $W^s_{\rm loc}(x)$, thus there exist
$w_0,w_1\in\{w\in E\cap W^s_{\rm loc}(x):[w,y]\in E\}$ with $d(w_0,w_1)<\delta$.
Take $w_2=[w_1,y]$, $w_3=[w_0,y]$.
\item $\gamma_0=\langle w_0,w_1,w_2,w_3,w_0\rangle$ is a $su$--loop with
$|P(\gamma_0)|<\frac{4c}{5}<c\Rightarrow P(\gamma_0)=0$.
\end{enumerate}
Let $\widehat\gamma_0$ be the lifted $su$--path of $\gamma_0$, and
$\gamma:=\pi_r(\widehat\gamma_0)$. Since $\pi_r:\Sigma_r^\#\to M$ is finite-to-one and $\nu^s_x, \nu^u_x$ have global support, we can choose $w_0,w_1$  so that $\text{diam}(\gamma)<\delta$.

We claim that if $\delta,\epsilon$ are small enough, then for every $|t|<\epsilon$, the quadrilateral $T^{t}\gamma$ is the boundary of a piecewise smooth {\em immersed} surface $T^t U$ s.t.:
\begin{enumerate}[$\circ$]
\item $T^tU$ is the union of  compact smooth {\em embedded} surfaces $T^tU_i$, $i=0,1,2,3$.
\item $T^tU_i$ are uniformly transverse to the Reeb vector field.
\item $T^tU_i$ have piecewise smooth boundaries and $\int_{T^t\gamma}=\int_{T^t(\partial U)}=\sum_{i=0}^3\int_{T^t(\partial U_i)}$.
\end{enumerate}

Had $\gamma$ been an euclidean rectangle, we could take $U$ to be its interior,
and $U_i$ the four triangles described by the principal diagonals. The general case is similar. It is enough to treat $t=0$,
since the case of small $t$ follows from uniform transversality.

Let $w_0,\ldots,w_3\in {_{-N}[}a_{-N},\ldots,a_N]$, where $N$ is large to be chosen later.
Let $u_0,\ldots,u_3$ be the vertices of $\widehat\gamma_0$. If $\delta$ is small enough, then $\gamma$ is covered by a chart of $M$ and we can think of $\vec{u}_i:=u_i,\vec{\gamma}(t):=\gamma(t)$ as vectors in $\R^3$.
If $N$ is sufficiently large, then $\vec{\gamma}'(t)$ is nearly parallel to $E^u(u_0)$ or $E^s(u_0)$ at all points. Therefore $\vec{\gamma}$ is made of four curves which are $C^1$ close to the sides of a parallelogram
s.t. $\vec{u}_1-\vec{u}_0$, $\vec{u}_2-\vec{u}_3$ are nearly parallel to $E^s(u_0)$ and
$\vec{u}_2-\vec{u}_1$, $\vec{u}_3-\vec{u}_0$ are nearly parallel to $E^u(u_0)$.
There is no loss of generality in assuming that these vectors have norm in $(\frac{1}{2},2)$. Let $\epsilon_0:=$
$C^1$ distance between $\gamma$ and a parallelogram with sides
$\vec{u}_1-\vec{u}_0$ and $\vec{u}_3-\vec{u}_0$. Then $\epsilon_0\to 0$ as $N\to\infty$.

Let $\vec{z}:=\frac{1}{4}(\vec{u}_0+\cdots+\vec{u}_4)$, then
$\vec{z}-\frac{1}{2}(\vec{u}_i+\vec{u}_{i+1})=\frac{1}{2}(\vec{u}_{i-1}-\vec{u}_{i})+O(\epsilon_0)$,
where $\vec{u}_i:=\vec{u}_{i\text{(mod 4)}}$ (the approximation is an identity for real parallelograms).
We define $U_i$ to be the  cone with vertex $z$ and base $\vec{\gamma}_i$,
where $\vec{\gamma}_i:[0,1]\to \R^3$ is the ``leg" of $\vec\gamma$ from $\vec{u}_i$ to $\vec{u}_{i+1}$:
$$
U_i:=\{\vec{x}_i(s,t):=s\vec{\gamma}_i(t)+(1-s)\vec{z}:s,t\in [0,1]\},\ \ i=0,\ldots,3.
$$
$U_i$ are embedded, and $\int_{\gamma}=\int_{\partial U}=\sum_{i=0}^3\int_{\partial U_i}$.
At $\vec{x}_i(s,t)$, $U_i$ is perpendicular  to
$$
\vec{n}=(\vec{\gamma}_i(t)-\vec{z})\times {\gamma}_i'(t)=\biggl(\vec{\gamma}_i(t)-\frac{\vec{u}_i+\vec{u}_{i+1}}{2}\biggr)\times\vec{\gamma}_i'(t)+\biggl(\frac{\vec{u}_i+\vec{u}_{i+1}}{2}-\vec{z}\biggr)\times\vec{\gamma}_i'(t).
$$
The first summand is $O(\epsilon_0|\vec{\gamma}'_i(t)|)$, being  the product of  vectors at angle $O(\epsilon_0)$.
The second summand is of size $\sim |\vec{\gamma}'_i(t)|$ and $\epsilon_0$--parallel
to $\vec{e}^u(u_0)\times\vec{e}^s(u_0)$. By Lemma \ref{lemma-quadrilaterals},
$U_i$ is almost parallel to $\ker(\alpha)$, whence uniformly transverse to the Reeb flow.

Fix $t_0>0$ so small that
$D_i:=\bigcup_{t\in[0,t_0]}T^tU_i$ is a flow box. So
$
0\neq \sum_{i=0}^3 \int_{D_i}\omega=\sum_{i=0}^3 \int_{0}^{t_0}\bigl(\int_{T^tU_i}i_X\omega\bigr)dt=\sum_{i=0}^3 \int_{0}^{t_0}\bigl(\int_{T^tU_i}d\alpha\bigr)dt=\int_{0}^{t_0}\bigl(\int_{T^tU}d\alpha\bigr)dt
$. But by the Stokes Theorem, this equals  $\int_{0}^{t_0}\bigl(\int_{T^t\gamma}\alpha\bigr)dt=0$, since the inner integral is zero by  Lemma \ref{lemma-quadrilaterals}. We obtain a contradiction. \end{proof}

\section{Equilibrium states for the geometric potential}\label{Section-Geometric-Potential}

Let $M$ be a three dimensional compact $C^\infty$ Riemannian manifold, let $X:M\to TM$ be a
non-vanishing $C^{1+\epsilon}$ vector field, and let $\textsf{T}$ be the flow on $M$ generated by $X$.
Throughout this section we assume $\textsf{T}$ has positive topological entropy.

\medskip
\noindent
{\sc The subset $M_{\rm hyp}\subset M$:} $p\in M_{\rm hyp}$ if there
are unit vectors $\un{e}^s_p,\un{e}^u_p\in T_pM$ s.t.
$$
\lim_{t\to\pm\infty} \frac{1}{|t|}\log\|dT^t_p\un{e}^s_p\|<0\ \text{ and } \ \lim_{t\to\pm\infty} \frac{1}{|t|}\log\|dT^t_p\un{e}^u_p\|>0.
$$

\medskip
If $\un{e}^s_p,\un{e}^u_p$ exist, then they are unique up to a sign, hence $M_{\rm hyp}$
is $\textsf{T}$--invariant.
By the Oseledets theorem and the Ruelle entropy inequality, any $\textsf{T}$--invariant and ergodic measure with 
positive metric entropy is carried by $M_{\rm hyp}$.

\medskip
\noindent
{\sc The geometric potential of $\textsf{T}$ \cite{Bowen-Ruelle-SRB}:}  $J:M_{\rm hyp}\to\R$ given by
$$
J(p)=-\tfrac{d}{dt}\Big|_{t=0}\log\|dT^t_p \un{e}^u_p\|=-\lim_{t\to 0}\tfrac{1}{t}\log\|dT^t_p\un{e}^u_p\|.
$$

\noindent
$J$ is bounded, since $\{T^t\}$ is $C^{1+\epsilon}$.

\subsection*{Some facts from \cite{Lima-Sarig}}

We recall some facts from \cite[\S2]{Lima-Sarig}. There exists a Poincar\'e section $\Lambda\subset M$
with return map $f:\Lambda\to\Lambda$ and roof function $R:\Lambda\to\R$ s.t.:
\begin{enumerate}[(1)]
\item $\Lambda$ is the union of disjoint discs transverse to $X$.
\item $\inf R>0$ and $\sup R<\infty$.
\item Let $\mathfrak{S}\subset\Lambda$ denote the {\em singular set} of $f:\Lambda\to\Lambda$, consisting of points $p$ which do not have a (relative) neighborhood $V\subset\Lambda\setminus\partial\Lambda$ which is diffeomorphic to a disc, such that $f|_V, f^{-1}|_V$ are diffeomorphisms onto their images. There is a constant $\mathfrak C$
s.t. $R,f,f^{-1}$ are differentiable on $\Lambda':=\Lambda\backslash\mathfrak{S}$
with $\sup_{p\in\Lambda'}\|dR_p\|<\mathfrak C$,
$\sup_{p\in\Lambda'}\|df_p\|<\mathfrak C$, $\sup_{p\in\Lambda'}\|(df_p)^{-1}\|<\mathfrak C$, and
$\|f|_{U}\|_{C^{1+\epsilon}}<\mathfrak C,\|f^{-1}|_{U}\|_{C^{1+\epsilon}}<\mathfrak C$ for all open
and connected $U\subset\Lambda'$. See \cite[Lemma 2.5]{Lima-Sarig}.
\item For all $p\in\Lambda_{\rm hyp}:=(\Lambda\backslash\bigcup_{n\in\Z}f^n(\mathfrak S))\cap M_{\rm hyp}$
there are $\vec{v}^s_p,\vec{v}^u_p\in T_p\Lambda$ unitary s.t.
$$
\lim_{n\to\pm\infty}\tfrac{1}{|n|}\log\|df^n_p\vec{v}^s_p\|<0\ \text{ and }
\ \lim_{n\to\pm\infty}\tfrac{1}{|n|}\log\|df^n_p\vec{v}^u_p\|>0.
$$
See \cite[Lemma 2.6]{Lima-Sarig} and its proof.
\end{enumerate}
Suppose $\mu$ is a hyperbolic $\textsf{T}$--invariant probability measure on $M$, and $\mu_\Lambda$, the ``induced measure", is the measure on $\Lambda$ s.t. 
$\mu=\frac{1}{\int_\Lambda Rd\mu_\Lambda}\int_\Lambda\left[\int_0^{R(p)}\delta_{T^tp}dt\right]d\mu_\Lambda(p)$. Then $\Lambda$ 
can be chosen with the additional properties below.
\begin{enumerate}[(5)]
\item[(5)]  The induced measure $\mu_\Lambda$ on $\Lambda$ satisfies:
\begin{enumerate}[(5.1)]
\item[(5.1)] $\mu_\Lambda(\mathfrak S)=0$.
\item[(5.2)] $\lim_{n\to\infty}\frac{1}{n}{\rm dist}_\Lambda(f^n(p),\mathfrak S)=0$ $\mu_\Lambda$--a.e.
\end{enumerate}
See \cite[Thm 2.8]{Lima-Sarig}.
\item[(6)] There are a TMF $\sigma_r:\Sigma_r\to\Sigma_r$ and
H\"older continuous maps $\pi:\Sigma\to \Lambda$ and $\pi_r:\Sigma_r\to M$ s.t.:
\begin{enumerate}[(6.1)]
\item[(6.1)] $\pi\circ \sigma=f\circ \pi$, $\pi[\Sigma^\#]$ has full $\mu_\Lambda$--measure, and every
$x\in \pi[\Sigma^\#]$ has finitely many pre-images in $\Sigma^\#$.
\item[(6.2)] $\pi_r(x,t)=T^t\pi(x)$, $\pi_r\circ \sigma_r=T\circ\pi_r$, $\pi_r[\Sigma_r^\#]$ has full $\mu$--measure,
and every $p\in\pi_r[\Sigma_r]$ has finitely many pre-images in $\Sigma_r^\#$.
\end{enumerate}
See \cite[Thm 5.6]{Lima-Sarig}. $\Sigma^\#,\Sigma^\#_r$ denote the regular parts of $\Sigma,\Sigma_r$,
see \cite[\S1]{Lima-Sarig}.
\end{enumerate}

\medskip
\noindent
Finally, if $\mu$ is ergodic with  positive entropy then $h_{\mu_\Lambda}(f)>0$ and
$\mu_\Lambda(\Lambda_{\rm hyp})=1$.

\medskip
\noindent
{\sc Geometric potential of $f:\Lambda\to\Lambda$:} 
$J^f:\Lambda_{\rm hyp}\to\R$, $J^f(p)=-\log\|df_p\vec{v}^u_p\|$.

\medskip
\noindent
$J^f$ is bounded, since $\sup_{p\in\Lambda'}\|df_p\|<\mathfrak C,\sup_{p\in\Lambda'}\|(df_p)^{-1}\|<\mathfrak C$.
Even though $J,J^f$ are not globally defined, we can define their equilibrium measures.

\medskip
\noindent
{\sc Equilibrium measures of $J$ and $J^f$:} $\mu$ is called an {\em equilibrium measure of} $J$
if $h_\mu(T^1)+\int J d\mu=P_{\rm top}(J)$, where
$$
P_{\rm top}(J):=\sup\left\{
h_\nu(T^1)+\int_M J d\nu:
\begin{array}{c}
\nu\text{ is $\textsf{T}$--invariant Borel probability}\\
\text{ measure with }\nu(M_{\rm hyp})=1
\end{array}
\right\}.
$$
$P_{\rm top}(J)$ is called the {\em topological pressure} of $J$. Similar definitions hold for $J^f$ with
$T^1,M_{\rm hyp}$ replaced by $f,\Lambda_{\rm hyp}$.

\medskip
\noindent
$P_{\rm top}(J),P_{\rm top}(J^f)<\infty$, since $J,J^f$ are bounded. Similar definitions can also be given 
for functions of the form $aJ,bf$, $a,b\in\R$.

\begin{lem}\label{Lemma-relation-geometric-potentials}
Assume that $\Lambda,f,R$ and $\mu$ satisfy conditions $(1)$--$(6)$ above.
Then $\mu$ is an equilibrium measure of $J$ iff
$\mu_\Lambda$ is an equilibrium measure of $J^f-P_{\rm top}(J)f$.
\end{lem}

\begin{proof}
Let $\overline{J}:\Lambda_{\rm hyp}\to\R$,
$\overline{J}(p)=\int_0^{R(p)}J(T^sp)ds$. As in claim 1 of the proof of Theorem \ref{Thm-RPF},
$\mu$ is an equilibrium measure of $J$ iff $\mu_\Lambda$ is an equilibrium measure of $\ov{J}-P_{\rm top}(J)f$. We will show that $\int_\Lambda \overline{J} d\nu=\int_\Lambda J^fd\nu$ for every
$f$--invariant $\nu$ with $\nu(\Lambda_{\rm hyp})=1$, and deduce that $\mu$ is an equilibrium measure of $J$ iff $\mu_\Lambda$ is an equilibrium measure of $J^f-P_{\rm top}(J)f$.

A simple calculation\footnote{Let $h(t):=-\log\|dT^t\un{e}^u_p\|$, then
$h(0)=0$ and $-\log\|dT^t\un{e}^u_{T^sp}\|=h(t+s)-h(s)$, therefore
$-\tfrac{d}{dt}\Big|_{t=0}\log\|dT^t\un{e}^u_{T^sp}\|=\tfrac{d}{dt}\Big|_{t=0}[h(t+s)-h(s)]=h'(s)$.
By the fundamental theorem of calculus, $\overline{J}(p)=\int_0^{R(p)}h'(s)ds=h(R(p))=-\log\|dT^{R(p)}\un{e}^u_p\|$.}
shows that
\begin{align*}
\overline{J}(p)=-\int_0^{R(p)}\tfrac{d}{dt}\Big|_{t=0}\log\|dT^t\un{e}^u_{T^sp}\|ds
=-\log\|dT^{R(p)}\un{e}^u_p\|.
\end{align*}
Since $f(p)=T^{R(p)}(p)$, we have $df_pv=dT^{R(p)}_pv+\<\nabla R(p),v\> X_{f(p)}$, $\forall v\in T_p\Lambda$. Write
$\vec{v}^u_p=\alpha(p)\un{e}^u_p+\beta(p)X_p$ (necessarily $\alpha(p)\neq0$). Then
\begin{align*}
df\vec{v}^u_p&=dT^{R(p)}\vec{v}^u_p+\<\nabla R(p),\vec{v}^u_p\>X_{f(p)}\\
&=\alpha(p)dT^{R(p)}\un{e}^u_p+\beta(p)dT^{R(p)}X_p+\<\nabla R(p),\vec{v}^u_p\>X_{f(p)}\\
&=\pm \alpha(p)\|dT^{R(p)}\un{e}^u_p\|\un{e}^u_{f(p)}+[\beta(p)+\<\nabla R(p),\vec{v}^u_p\>]X_{f(p)}.
\end{align*}
 Similarly
$
df\vec{v}^u_p=\pm \|df\vec{v}^u_p\|\vec{v}^u_{f(p)}=\pm(\alpha(f(p))\|df\vec{v}^u_p\|\un{e}^u_{f(p)}+
\beta(f(p))\|df\vec{v}^u_p\|X_{f(p)}).
$

Comparing the $\un{e}^u_{f(p)}$ components, we get
$|\alpha(p)|\|dT^{R(p)}\un{e}^u_p\|=|\alpha(f(p))|\|df\vec{v}^u_p\|$. Hence $U:\Lambda_{\rm hyp}\to\R$,
$U(p):=\log|\alpha(p)|$ is a measurable function with $$\overline{J}=J^f+U\circ f-U.$$ 

We use this to show that $\mu$ is an equilibrium measure for $J$ iff $\mu_\Lambda$ is an equilibrium measure for $J^f-P_{\rm top}(J)f$.
By (4) and (5), $\mu(M_{\rm hyp})=\mu_\Lambda(\Lambda_{\rm hyp})$.
Let $\nu$ be an ergodic $f$--invariant
probability measure with $\nu(\Lambda_{\rm hyp})=1$.
By the Birkhoff ergodic theorem, $\lim_{n\to\infty}\frac{1}{n}\overline{J}_n=\int_\Lambda\overline{J} d\nu$
and $\lim_{n\to\infty}\frac{1}{n}J_n^f=\int_\Lambda J^f d\nu$ $\nu$--a.e.
By the Poincar\'e recurrence theorem, $\liminf_{n\to\infty}|U(f^n(p))-U(p)|<\infty$
$\nu$--a.e., hence for $\nu$--a.e. $p\in\Lambda$ we have
$\int_\Lambda\overline{J}d\nu=\liminf_{n\to\infty}\tfrac{1}{n}\overline{J}(p)=\liminf_{n\to\infty}\tfrac{1}{n}J^f(p)
=\int_\Lambda J^f d\nu$.
By the ergodic decomposition, $\int_\Lambda\overline{J}d\nu=\int_\Lambda J^f d\nu$ for every $f$--invariant $\nu$ s.t.
$\nu(\Lambda_{\rm hyp})=1$. The lemma follows from the discussion at the beginning of the proof.
\end{proof}

\begin{lem}
$[J^f-P_{\rm top}(J)f]\circ\pi$
is a H\"older continuous potential on $\Sigma$ with respect to the symbolic metric.
\end{lem}
\begin{proof}
$f\circ\pi:\Sigma\to\R$ is H\"older  by construction: $f\circ\pi=r$ and roof functions of TMF are H\"older. 
$J^f\circ\pi$ is H\"older, because $df$ is uniformly H\"older on $\Lambda'$ and $x\in\Sigma\to\vec{v}^u_{\pi(x)}$ is H\"older  by 
\cite[Lemma 5.7]{Lima-Sarig}. 
\end{proof}

\begin{proof}[Proof of Theorem \ref{Thm-Geometric-Potential}.]
Fix $\chi>0$, and let $\mu$ be a $\chi$--hyperbolic\footnote{$\mu$ is $\chi$--hyperbolic if $\mu$--a.e.
point has one Lyapunov exponent $>\chi$ and another $<-\chi$.} equilibrium measure of $J$ with $h_\mu(T^1)>0$.
Take $\Lambda,f,R$ satisfying $(1)$--$(6)$ above. Since $\mu$ is carried by $M_{\rm hyp}$,
Lemma \ref{Lemma-relation-geometric-potentials} implies that $\mu_\Lambda$ is an
equilibrium measure of $J^f-P_{\rm top}(J)f$. Arguing as in \cite[Theorem 6.2]{Lima-Sarig},
the function $[J^f-P_{\rm top}(J)f]\circ\pi:\Sigma\to\R$ has an equilibrium measure $\widehat{\mu_\Lambda}$
s.t. $\widehat{\mu_\Lambda}\circ\pi^{-1}=\mu_\Lambda$.
The potential $[J^f-P_{\rm top}(J)f]\circ\pi$
is H\"older continuous. Since ergodic equilibrium measures of H\"older potentials
on a TMS are carried by topologically transitive TMS, $\widehat{\mu_\Lambda}$ has at most countably many
ergodic components. This shows that $J$ has at most countably many $\chi$--hyperbolic ergodic equilibrium measures:
if there were uncountably many, then their convex
combination would generate a $\chi$--hyperbolic equilibrium measure with uncountably many ergodic
components.

Assume now that $\mu$ is also ergodic. We can choose $\widehat{\mu_\Lambda}$ to be ergodic.
The measure $\widehat{\mu_\Lambda}$ is the induced measure of some $\widehat\mu$ on $\Sigma_r$,
hence $\widehat\mu\circ\pi_r^{-1}=\mu$.
By Theorems \ref{Thm-K-Mixing} and \ref{Thm-Bernoulli} $(\Sigma_r,\widehat\mu,\sigma_r)$ is
Bernoulli up to a period. Since $\widehat\mu$ projects to $\mu$, $(M,\mu,\{T^t\})$
is also Bernoulli up to a period.

If additionally $\textsf{T}$ is a Reeb flow, then $\sigma_r$ is Bernoulli and
so is $\textsf{T}$.
\end{proof}

\section{Acknowledgements}
We thank Federico Rodriguez-Hertz for pointing out \cite{Katok-Burns-Infinitesimal}.
We also thank J\'er\^{o}me Buzzi, Yakov Pesin, and the referee for suggesting we extend our results to scalar
multiples of the geometric potential.

\bibliographystyle{alpha}
\bibliography{Bernoulli-for-Flows-bib}{}

\def\cprime{$'$} \def\cprime{$'$} \def\cprime{$'$} \def\cprime{$'$}
\begin{thebibliography}{KSLP86}

\bibitem[Abr59]{Abramov-Entropy-Flows}
L.~M. Abramov.
\newblock On the entropy of a flow.
\newblock {\em Dokl. Akad. Nauk SSSR}, 128:873--875, 1959.

\bibitem[AM78]{Abraham-Marsden-Mechanics}
Ralph Abraham and Jerrold~E. Marsden.
\newblock {\em Foundations of mechanics}.
\newblock Benjamin/Cummings Publishing Co., Inc., Advanced Book Program,
  Reading, Mass., 1978.
\newblock Second edition, revised and enlarged, With the assistance of Tudor
  Ra{\c{t}}iu and Richard Cushman.

\bibitem[BBE85]{Ballmann-Brin-Eberlein}
Werner Ballmann, Misha Brin, and Patrick Eberlein.
\newblock Structure of manifolds of nonpositive curvature. {I}.
\newblock {\em Ann. of Math. (2)}, 122(1):171--203, 1985.

\bibitem[BG89]{Burns-Gerber}
Keith Burns and Marlies Gerber.
\newblock Real analytic {B}ernoulli geodesic flows on {$S^2$}.
\newblock {\em Ergodic Theory Dynam. Systems}, 9(1):27--45, 1989.

\bibitem[BM77]{Bowen-Marcus}
Rufus Bowen and Brian Marcus.
\newblock Unique ergodicity for horocycle foliations.
\newblock {\em Israel J. Math.}, 26(1):43--67, 1977.

\bibitem[Bow75]{Bowen-LNM}
Rufus Bowen.
\newblock {\em Equilibrium states and the ergodic theory of {A}nosov
  diffeomorphisms}.
\newblock Lecture Notes in Mathematics, Vol. 470. Springer-Verlag, Berlin,
  1975.

\bibitem[Bow75]{Bowen-Bernoulli}
Rufus Bowen.
\newblock Bernoulli equilibrium states for {A}xiom {A} diffeomorphisms.
\newblock {\em Math. Systems Theory}, 8(4):289--294, 1974/75.

\bibitem[BP07]{Barreira-Pesin-Non-Uniform-Hyperbolicity-Book}
Luis Barreira and Yakov Pesin.
\newblock {\em Nonuniform hyperbolicity}, volume 115 of {\em Encyclopedia of
  Mathematics and its Applications}.
\newblock Cambridge University Press, Cambridge, 2007.
\newblock Dynamics of systems with nonzero Lyapunov exponents.

\bibitem[BR75]{Bowen-Ruelle-SRB}
Rufus Bowen and David Ruelle.
\newblock The ergodic theory of {A}xiom {A} flows.
\newblock {\em Invent. Math.}, 29(3):181--202, 1975.

\bibitem[Bri75]{Brin-Group-Extensions}
M.~I. Brin.
\newblock The topology of group extensions of {$C$}-systems.
\newblock {\em Mat. Zametki}, 18(3):453--465, 1975.

\bibitem[BS03]{Buzzi-Sarig}
J{\'e}r{\^o}me Buzzi and Omri Sarig.
\newblock Uniqueness of equilibrium measures for countable {M}arkov shifts and
  multidimensional piecewise expanding maps.
\newblock {\em Ergodic Theory Dynam. Systems}, 23(5):1383--1400, 2003.

\bibitem[BW72]{Bowen-Walters-Metric}
Rufus Bowen and Peter Walters.
\newblock Expansive one-parameter flows.
\newblock {\em J. Diff. Equations}, 12:180--193, 1972.

\bibitem[CFS82]{Cornfeld-Fomin-Sinai}
I.~P. Cornfeld, S.~V. Fomin, and Ya.~G. Sina{\u\i}.
\newblock {\em Ergodic theory}, volume 245 of {\em Grundlehren der
  Mathematischen Wissenschaften [Fundamental Principles of Mathematical
  Sciences]}.
\newblock Springer-Verlag, New York, 1982.

\bibitem[Dao13]{Daon}
Yair Daon.
\newblock Bernoullicity of equilibrium measures on countable {M}arkov shifts.
\newblock {\em Discrete Contin. Dyn. Syst.}, 33(9):4003--4015, 2013.

\bibitem[Gei08]{Geiges-Contact-Topology}
Hansj{\"o}rg Geiges.
\newblock {\em An introduction to contact topology}, volume 109 of {\em
  Cambridge Studies in Advanced Mathematics}.
\newblock Cambridge University Press, Cambridge, 2008.

\bibitem[GH88]{Guivarch-Hardy}
Y.~Guivarc'h and J.~Hardy.
\newblock Th\'eor\`emes limites pour une classe de cha\^\i nes de {M}arkov et
  applications aux diff\'eomorphismes d'{A}nosov.
\newblock {\em Ann. Inst. H. Poincar\'e Probab. Statist.}, 24(1):73--98, 1988.

\bibitem[Gur67]{Gurevich-K-flows}
B.~M. Gurevi{\v{c}}.
\newblock Certain conditions for the existence of {$K$}-decompositions for
  special flows.
\newblock {\em Trudy Moskov. Mat. Ob\v s\v c.}, 17:89--116, 1967.

\bibitem[HPT04]{Hu-Pesin-Talitskaya-2004}
H.~Hu, Y.~Pesin, and A.~Talitskaya.
\newblock Every compact manifold carries a hyperbolic {B}ernoulli flow.
\newblock In {\em Modern {D}ynamical {S}ystems and {A}pplications}, pages
  347--358. Cambridge Univ. Press, Cambridge, 2004.

\bibitem[Kat82]{Katok-Closed-Geodesics}
A.~Katok.
\newblock Entropy and closed geodesics.
\newblock {\em Ergodic Theory Dynam. Systems}, 2(3-4):339--365 (1983), 1982.

\bibitem[Kat94]{Katok-Burns-Infinitesimal}
A.~Katok.
\newblock Infinitesimal {L}yapunov functions, invariant cone families and
  stochastic properties of smooth dynamical systems.
\newblock {\em Ergodic Theory Dynam. Systems}, 14(4):757--785, 1994.
\newblock With the collaboration of K. Burns.

\bibitem[Kea72]{Keane-g-measures}
Michael Keane.
\newblock Strongly mixing {$g$}-measures.
\newblock {\em Invent. Math.}, 16:309--324, 1972.

\bibitem[Kit98]{Kitchens-Book}
Bruce~P. Kitchens.
\newblock {\em Symbolic dynamics}.
\newblock Universitext. Springer-Verlag, Berlin, 1998.
\newblock One-sided, two-sided and countable state Markov shifts.

\bibitem[Kni98]{Knieper-Rank-One-Entropy}
Gerhard Knieper.
\newblock The uniqueness of the measure of maximal entropy for geodesic flows
  on rank {$1$} manifolds.
\newblock {\em Ann. of Math. (2)}, 148(1):291--314, 1998.

\bibitem[KSLP86]{Katok-Strelcyn}
Anatole Katok, Jean-Marie Strelcyn, F.~Ledrappier, and F.~Przytycki.
\newblock {\em Invariant manifolds, entropy and billiards; smooth maps with
  singularities}, volume 1222 of {\em Lecture Notes in Mathematics}.
\newblock Springer-Verlag, Berlin, 1986.

\bibitem[Led74]{Ledrappier-Principe-Variationnel}
F.~Ledrappier.
\newblock Principe variationnel et syst\`emes dynamiques symboliques.
\newblock {\em Z. Wahrscheinlichkeitstheorie und Verw. Gebiete}, 30:185--202,
  1974.

\bibitem[Led84]{Ledrappier-Mesures-de-Sinai}
F.~Ledrappier.
\newblock Propri\'et\'es ergodiques des mesures de {S}ina\"\i.
\newblock {\em Inst. Hautes \'Etudes Sci. Publ. Math.}, (59):163--188, 1984.

\bibitem[LR69]{Lanford-Ruelle}
O.~E. Lanford, III and D.~Ruelle.
\newblock Observables at infinity and states with short range correlations in
  statistical mechanics.
\newblock {\em Comm. Math. Phys.}, 13:194--215, 1969.

\bibitem[LS]{Lima-Sarig}
Yuri Lima and Omri Sarig.
\newblock Symbolic dynamics for three dimensional flows with positive entropy.
\newblock Preprint.

\bibitem[New89]{Newhouse-Entropy}
Sheldon~E. Newhouse.
\newblock Continuity properties of entropy.
\newblock {\em Ann. of Math. (2)}, 129(2):215--235, 1989.

\bibitem[Orn70a]{Ornstein-Imbedding-in-Flows}
D.~S. Ornstein.
\newblock Imbedding {B}ernoulli shifts in flows.
\newblock In {\em Contributions to {E}rgodic {T}heory and {P}robability
  ({P}roc. {C}onf., {O}hio {S}tate {U}niv., {C}olumbus, {O}hio, 1970)}, pages
  178--218. Springer, Berlin, 1970.

\bibitem[Orn70b]{Ornstein-Isomorphism-Thm}
Donald Ornstein.
\newblock Bernoulli shifts with the same entropy are isomorphic.
\newblock {\em Advances in Math.}, 4:337--352 (1970), 1970.

\bibitem[Orn70c]{Ornstein-Factor-Bernoulli}
Donald Ornstein.
\newblock Factors of {B}ernoulli shifts are {B}ernoulli shifts.
\newblock {\em Advances in Math.}, 5:349--364 (1970), 1970.

\bibitem[Orn73]{Ornstein-B-Flows}
Donald~S. Ornstein.
\newblock The isomorphism theorem for {B}ernoulli flows.
\newblock {\em Advances in Math.}, 10:124--142, 1973.

\bibitem[Orn74]{Ornstein-Book}
Donald~S. Ornstein.
\newblock {\em Ergodic theory, randomness, and dynamical systems}.
\newblock Yale University Press, New Haven, Conn.-London, 1974.
\newblock James K. Whittemore Lectures in Mathematics given at Yale University,
  Yale Mathematical Monographs, No. 5.

\bibitem[OW73]{Ornstein-Weiss-Geodesic-Flows}
Donald~S. Ornstein and Benjamin Weiss.
\newblock Geodesic flows are {B}ernoullian.
\newblock {\em Israel J. Math.}, 14:184--198, 1973.

\bibitem[OW98]{Ornstein-Weiss-General-Bernoulli}
Donald Ornstein and Benjamin Weiss.
\newblock On the {B}ernoulli nature of systems with some hyperbolic structure.
\newblock {\em Ergodic Theory Dynam. Systems}, 18(2):441--456, 1998.

\bibitem[Pes76]{Pesin-Izvestia-1976}
Ja.~B. Pesin.
\newblock Families of invariant manifolds that correspond to nonzero
  characteristic exponents.
\newblock {\em Izv. Akad. Nauk SSSR Ser. Mat.}, 40(6):1332--1379, 1440, 1976.

\bibitem[Pes77]{Pesin-Characteristic-1977}
Ja.~B. Pesin.
\newblock Characteristic {L}japunov exponents, and smooth ergodic theory.
\newblock {\em Uspehi Mat. Nauk}, 32(4 (196)):55--112, 287, 1977.

\bibitem[Pes78]{Pesin-Entropy-Formulas}
Ja.~B. Pesin.
\newblock Equations for the entropy of a geodesic flow on a compact riemannian
  manifold without conjugate points.
\newblock {\em Mathematical notes of the Academy of Sciences of the USSR},
  24(4):796--805, 1978.

\bibitem[Pin60]{Pinsker}
M.~S. Pinsker.
\newblock Dynamical systems with completely positive or zero entropy.
\newblock {\em Soviet Math. Dokl.}, 1:937--938, 1960.

\bibitem[PP90]{Parry-Pollicott-Asterisque}
William Parry and Mark Pollicott.
\newblock Zeta functions and the periodic orbit structure of hyperbolic
  dynamics.
\newblock {\em Ast\'erisque}, (187-188):268, 1990.

\bibitem[Rat73]{Ratner-MP-n-dimensions}
M.~Ratner.
\newblock Markov partitions for {A}nosov flows on {$n$}-dimensional manifolds.
\newblock {\em Israel J. Math.}, 15:92--114, 1973.

\bibitem[Rat74]{Ratner-Flows-Bernoulli}
M.~Ratner.
\newblock Anosov flows with {G}ibbs measures are also {B}ernoullian.
\newblock {\em Israel J. Math.}, 17:380--391, 1974.

\bibitem[Rat78]{Ratner-Flows-K}
M.~Ratner.
\newblock Bernoulli flow over maps of the interval.
\newblock {\em Israel J. Math.}, 31(3-4):298--314, 1978.

\bibitem[Roh61]{Rohlin-Exactness}
V.~A. Rohlin.
\newblock Exact endomorphisms of a {L}ebesgue space.
\newblock {\em Izv. Akad. Nauk SSSR Ser. Mat.}, 25:499--530, 1961.

\bibitem[RS61]{Rohlin-Sinai}
V.~A. Rohlin and Ja.~G. Sina{\u\i}.
\newblock The structure and properties of invariant measurable partitions.
\newblock {\em Dokl. Akad. Nauk SSSR}, 141:1038--1041, 1961.

\bibitem[Rue78a]{Ruelle-Entropy-Inequality}
David Ruelle.
\newblock An inequality for the entropy of differentiable maps.
\newblock {\em Bol. Soc. Brasil. Mat.}, 9(1):83--87, 1978.

\bibitem[Rue78b]{Ruelle-TDF-book}
David Ruelle.
\newblock {\em Thermodynamic formalism}, volume~5 of {\em Encyclopedia of
  Mathematics and its Applications}.
\newblock Addison-Wesley Publishing Co., Reading, Mass., 1978.
\newblock The mathematical structures of classical equilibrium statistical
  mechanics, With a foreword by Giovanni Gallavotti and Gian-Carlo Rota.

\bibitem[Sar01]{Sarig-Null-Recurrent}
Omri~M. Sarig.
\newblock Thermodynamic formalism for null recurrent potentials.
\newblock {\em Israel J. Math.}, 121:285--311, 2001.

\bibitem[Sar11]{Sarig-Bernoulli-JMD}
Omri~M. Sarig.
\newblock Bernoulli equilibrium states for surface diffeomorphisms.
\newblock {\em J. Mod. Dyn.}, 5(3):593--608, 2011.

\bibitem[Sin72]{Sinai-Gibbs}
Ja.~G. Sina{\u\i}.
\newblock Gibbs measures in ergodic theory.
\newblock {\em Uspehi Mat. Nauk}, 27(4(166)):21--64, 1972.

\bibitem[Tho02]{Thouvenot-Handbook}
Jean-Paul Thouvenot.
\newblock Entropy, isomorphism and equivalence in ergodic theory.
\newblock In {\em Handbook of dynamical systems, {V}ol.\ 1{A}}, pages 205--238.
  North-Holland, Amsterdam, 2002.

\end{thebibliography}

\end{document}